\documentclass[english,reqno]{amsart}

\usepackage{amsmath}
\usepackage{amssymb} 
\usepackage{enumitem} 
\usepackage{mathtools} 
\usepackage[table]{xcolor} 
\usepackage[all]{xy} 
\usepackage{tikz} 
\usepackage{tikz-cd}
\usepackage{indentfirst} 
\usepackage{babel} 
\usepackage{setspace} 

\usepackage[colorlinks,linkcolor=red,anchorcolor=green,citecolor=blue]{hyperref} 
\hypersetup{linktocpage = true} 

\usepackage{rotating} 

\usepackage{ytableau} 
\usepackage{longtable} 
\newcolumntype{M}[1]{>{\centering\arraybackslash}m{#1}} 

\usepackage{mathpazo}
\usepackage{mathrsfs}
\DeclareFontFamily{OMS}{rsfs}{\skewchar\font'60}
\DeclareFontShape{OMS}{rsfs}{m}{n}{<-5>rsfs5 <5-7>rsfs7 <7->rsfs10 }{}
\DeclareSymbolFont{rsfs}{OMS}{rsfs}{m}{n}
\DeclareSymbolFontAlphabet{\scr}{rsfs}
\DeclareSymbolFontAlphabet{\scr}{rsfs}

\usepackage[T1]{fontenc}

\newcommand\cH{{\mathcal H}}

\newcommand\cV{{\mathcal V}}

\newcommand\bbB{{\mathbb B}}
\newcommand\bbC{{\mathbb C}}

\newcommand\bbK{{\mathbb K}}

\newcommand\bbP{{\mathbb P}}
\newcommand\bbQ{{\mathbb Q}}

\newcommand\bbS{{\mathbb S}}

\newcommand\sE{{\mathscr E}}
\newcommand\sF{{\mathscr F}}
\newcommand\sG{{\mathscr G}}
\newcommand\sI{{\mathscr I}}
\newcommand\sL{{\mathscr L}}
\newcommand\sM{{\mathscr M}}
\newcommand\sN{{\mathscr N}}
\newcommand\sO{{\mathscr O}}
\newcommand\sQ{{\mathscr Q}}
\newcommand\sT{{\mathscr T}}
\newcommand\sV{{\mathscr V}}
\newcommand\sH{{\mathscr H}}

\newcommand{\id}{{\rm id}}
\newcommand{\rk}{{\rm rk}}
\newcommand{\codim}{{\rm codim}}

\newcommand{\im}{{\rm im}}

\DeclareMathOperator*{\Sym}{Sym}

\DeclareMathOperator*{\supp}{Supp}

\DeclareMathOperator*{\PGL}{PGL}

\newcommand{\Chow}[1]{\ensuremath{\mbox{\rm Chow}(#1)}}

\newcommand{\RC}[1]{\ensuremath{\mbox{\rm RatCurves}^n(#1)}}

\newcommand{\Proj}[1]{\ensuremath{\mbox{\rm Proj}(#1)}}

\newcommand{\sHom}[2]{\ensuremath{\mathscr{H}om_{\mathscr{O}_X}(#1,#2)}}

\newcommand{\defeq}{{\vcentcolon=}}

\newcommand\bp{{\bar\partial}}

\theoremstyle{plain}
\newtheorem{thm}{Theorem}[section]
\newtheorem{lemma}[thm]{Lemma}
\newtheorem{prop}[thm]{Proposition}
\newtheorem{cor}[thm]{Corollary}
\newtheorem{defn}[thm]{Definition}
\newtheorem{claim}[thm]{Claim}

\theoremstyle{definition}
\newtheorem{example}[thm]{Example}
\newtheorem{remark}[thm]{Remark}

\newtheorem{convention}[thm]{Convention}

\setlist[itemize]{leftmargin=*}
\setlist[enumerate]{leftmargin=*}

\numberwithin{equation}{section} 

\makeatletter

\ifnum\@ptsize=0  \fi \ifnum\@ptsize=2  \fi

\title{Projective manifolds whose tangent bundle contains a strictly nef subsheaf}
\date{\today}

\subjclass[2010]{14H30,14J40,14J60,32Q57}
\keywords{strictly nef, MRC fibration,  numerically projectively flat, hyperbolicity}

\author{Jie Liu}

\address{Jie Liu, Morningside Center of Mathematics, Academy of Mathematics and Systems Science, Chinese Academy of Sciences, Beijing, 100190, China}
\email{jliu@amss.ac.cn}

\author{Wenhao Ou}

\address{Wenhao Ou, Institute of Mathematics, Academy of Mathematics and Systems Science, Chinese Academy of Sciences, Beijing, 100190, China}
\email{wenhaoou@amss.ac.cn}

\author{Xiaokui Yang}

\address{Xiaokui Yang, Department of Mathematics and Yau Mathematical Sciences Center, Tsinghua University, Beijing, 100084, China}
\email{xkyang@mail.tsinghua.edu.cn}

\begin{document}

\begin{abstract}Suppose that $X$ is a projective manifold whose tangent bundle $T_X$ contains a locally free strictly nef  subsheaf.  We prove that $X$ is isomorphic to a projective bundle over a hyperbolic  manifold.  Moreover, if  the fundamental group $\pi_1(X)$ is  virtually abelian, then $X$ is isomorphic to a projective space.
\end{abstract}

\maketitle

\tableofcontents

\vspace{-0.2cm}

\section{Introduction}

Since the seminal works of Mori   and Siu-Yau on the solutions to
Hartshorne
 conjecture   and Frankel conjecture
(\cite{Mori1979}, \cite{SiuYau1980}), it becomes apparent that the
positivity
  of the tangent bundle of a complex projective manifold
  carries important geometric information. In the past
decades,  many remarkable generalizations  have been
established. For instance, Mok classified compact K\"ahler manifold
with semipositive holomorphic bisectional curvature in
\cite{Mok1988}. Initiated by the fundamental works of Campana,
Demailly, Peternell and Schneider (\cite{CampanaPeternell1991},
\cite{DemaillyPeternellSchneider1994}, \cite{Peternell1996}), the
structure of projective manifolds with nef tangent bundles is
investigated by many mathematicians.
The  last building block to be understood for such manifolds are Fano manifolds with nef tangent bundles.
 Campana and Peternell proposed in \cite{CampanaPeternell1991} the following
conjecture, which is still an important  open problem: a  Fano manifold with nef tangent bundle  must be a rational
homogeneous space.
It is proved for all Fano manifolds of
dimension at most five and has also been verified for certain
special varieties.  We refer to
\cite{CampanaPeternell1991,CampanaPeternell1993,Mok2002,Hwang2006,Pandharipande2013,
Watanabe2014,MunozOcchettaSolaCondeWatanabeEtAl2015,Kanemitsu2017,Li2017}
and the references therein.

Recall that  a line bundle $\sL$ on a projective variety $X$  is said to be
\emph{strictly nef} if $c_1(\sL)\cdot C>0$ for all complete curves $C$ in $X$,
and a vector bundle $\sF$   is strictly nef if its tautological line
bundle $\sO_{\mathbb{P}(\sF)}(1)$ is strictly nef. The definition of strict nefness is quite natural and it is a notion of positivity which is stronger than nefness but weaker than ampleness.  The main difficulty to deal with it is that the strictly nefness is not closed under exterior product.  Actually, there exist \emph{Hermitian flat} vector bundles which are also strictly nef, and this phenomenen will be studied intensively in this paper.
Even though there are significant differences between strict nefness and ampleness, we  still expect that the strict nefness could play similar roles as ampleness in many  situations. Indeed,  together with Li, the second and third authors obtained the following theorem in \cite{LiOuYang2019} which extends Mori's result.

\begin{thm}\cite[Theorem 1.4]{LiOuYang2019}
    \label{thm:Li-Ou-Yang}
    Let $X$ be an $n$-dimensional  complex projective manifold such that $T_X$ is strictly nef. Then $X\cong  \bbP^n$.
\end{thm}

Meanwhile, it is also known that the existence of positive subsheaves of the
tangent bundle already impose strong geometric restrictions on the
ambient manifold. For example,  Andreatta and Wi\'sniewski achieved
the following characterization of projective spaces.
\begin{thm}\cite[Theorem]{AndreattaWisniewski2001}
    \label{thm:Andreatta-Wisniewski}
    Let $X$ be an $n$-dimensional   complex projective manifold.
    If there exists   a rank $r$ ample locally free subsheaf $\sF$ of $T_X$, then $X\cong  \bbP^n$ and either $\sF\cong  T_{\bbP^n}$  or $\sF\cong \sO_{\bbP^n}(1)^{\oplus r}$.
\end{thm}

\noindent When $\sF$ is a line bundle, this theorem is settled by
Wahl in \cite{Wahl1983} via the theory of algebraic derivations in
characteristic zero. In \cite{CampanaPeternell1998}, Campana and
Peternell proved the theorem in the cases $r\geqslant n-2$. Later,
it is shown that the assumption on the  local freeness can be
dropped, see \cite{AproduKebekusPeternell2008,Liu2019}.

In view of Theorem \ref{thm:Li-Ou-Yang}, it is natural to ask whether  Theorem \ref{thm:Andreatta-Wisniewski}
still holds  if the subsheaf $\sF$ is only assumed to be strictly nef.  Unfortunately,  Mumford
constructed an example (see \cite[Chapter I, Example 10.6]{Hartshorne1970}) which gives a negative answer to this
question. Indeed,   for any smooth projective curve $C$
of genus $g\geqslant 2$,
    there exists a rank
    $2$ Hermitian flat and strictly nef vector bundle $\sE$ over $C$. Then  the relative tangent bundle $T_{\bbP(\sE)/C}$ is a strictly nef subbundle of $T_{\bbP(\sE)}$ since it is isomorphic to the line bundle $\sO_{\bbP(\sE)}(2)$.

In this paper, we investigate the geometry of projective
manifolds whose tangent bundle contains a strictly nef subsheaf, and
obtain the following structure theorem, which is an extension of
Theorem \ref{thm:Andreatta-Wisniewski}.

\begin{thm}
    \label{thm:main-theorem}
    Let $X$ be a  complex projective manifold. Assume that the tangent bundle $T_X$ contains a locally free strictly nef subsheaf $\sF$ of rank $r>0$.
 Then $X$ admits a $\bbP^d$-bundle structure $\varphi\colon X\rightarrow T$ for some integer $d\geq r$.
 Furthermore
$T$ is a hyperbolic projective manifold.
\end{thm}

\noindent Here we recall that  a single point is also considered to be hyperbolic in the sense that every holomorphic map from $\bbC$ to it is a constant map.  Indeed,  we  obtain in Theorem \ref{cor:main:part1} a more concrete description on the
structure of the subsheaf $\sF$, and there are only two possibilities, which correspond to those of Theorem
\ref{thm:Andreatta-Wisniewski}:
    \begin{enumerate}
        \item $\sF\cong  T_{X/T}$ and $X$ is isomorphic to a flat projective bundle over
        $T$;
        \item $\sF$ is a numerically projectively flat vector bundle and its restriction  on every fiber of $\varphi$ is  isomorphic to  $\sO_{\bbP^d}(1)^{\oplus r}$.
    \end{enumerate}

\noindent  When $\sF$ is a line bundle, Druel obtained in
\cite{Druel2004} that $X$ is isomorphic to  either a projective
space or  a $\bbP^1$-bundle. However, when $\sF$ has greater rank,
 there are extra
structures as shown  in the second case above  (see Example
\ref{example:not-subbundles} for more details). To classify these structures,
 different methods are needed and  transcendental tools are
also crucially involved.

As an application of Theorem \ref{thm:main-theorem}, we obtain a  new characterization of projective spaces.

\begin{thm}
    \label{thm:simply-connected-Pn}  Let $X$ be an $n$-dimensional  complex projective manifold such that $T_X$ contains a locally free strictly nef subsheaf $\sF$.
    If $\pi_1(X)$ is virtually abelian, then $X$ is isomorphic to  $\bbP^n$, and $\sF$ is isomorphic to either $T_{\bbP^n}$ or $\sO_{\bbP^n}(1)^{\oplus r}$.
\end{thm}

 By using Theorem \ref{thm:main-theorem} and \cite[Theorem~0.1]{BrunebarbeKlinglerTotaro2013}, we get the existence of non-zero symmetric differentials.

\begin{cor}
    \label{cor:existence-symmetric-forms}
    Let $X$ be an $n$-dimensional projective manifold whose tangent bundle contains a  locally free strictly nef subsheaf. If $X$ is not isomorphic to $\bbP^n$, then $X$ has a non-zero symmetric differential, i.e. $H^0(X,\Sym^i\Omega_X)\not=0$ for some $i>0$.
\end{cor}

One of the important tools for the proof of Theorem
\ref{thm:main-theorem} is the theory of numerically projectively
flat vector bundles. The following criterion is  a variant of \cite[Theorem 1.8]{HoeringPeternell2019}, which is also the first step to understand the structures of projective bundles induced by strictly nef subsheaves.

\begin{thm}
    \label{Num-Projectily-flatness-criterion}
    Let $X$ be an $n$-dimensional complex projective manifold and  $\sF$ be a reflexive coherent sheaf of rank $r$. Assume that there exists a $\bbQ$-Cartier divisor class $\delta\in N^1(X)_{\bbQ}$  such that the $\mathbb{Q}$-twisted coherent sheaf $\sF\hspace{-0.8ex}<\hspace{-0.8ex}\delta\hspace{-0.8ex}>$ is almost nef and that
    \[
    (c_1(\sF)+r\delta)\cdot A^{n-1}=0
    \]
    for some ample divisor $A$. Then $\sF$ is locally free and numerically projectively flat.
\end{thm}

It is known that  numerically projectively flat vector bundles are
extensions of \emph{projectively Hermitian flat} vector bundles. In
this paper, we show that they are actually isomorphic to
projectively flat vector bundles with compatible holomorphic
connections, and this structure is crucial in the proof of Theorem
\ref{thm:main-theorem}. We refer the reader to \cite[Section 3]{Simpson1992} (see also \cite{Deng2018}) for a similar result on numerically flat vector bundles.

\begin{thm}\label{thm:num-proj-flat=proj-flat}
    Let $X$ be a projective manifold and $\sE$ be a numerically projectively flat vector bundle on $X$. Then $\sE$ is isomorphic to a  projectively flat holomorphic vector bundle
    $\sF$, i.e. there exists a projectively flat connection $\nabla$
    on $\sF$ such that  $\nabla^{0,1}=\bar\partial_{\sF}$ where
    $\nabla^{0,1}$ is the $(0,1)$-part of $\nabla$.
\end{thm}

The paper is organized as follows. After giving some elementary
results  in Section
\ref{section:pre}, we  recall the basics of $\mathbb{Q}$-twisted
sheaves in Section \ref{section:Q-twisted}. We prove
Theorem \ref{Num-Projectily-flatness-criterion} and Theorem
\ref{thm:num-proj-flat=proj-flat} in Section
\ref{section:proj-flat}. In Section \ref{Examples}, we  study
a  special case when $X$ is of the form
$\mathbb{P}(\sE)$  and  include some examples. Starting from Section
\ref{section:uniruled}, we  focus on the proof of Theorem
\ref{thm:main-theorem} in the general setting. We first show that
the MRC fibration of $X$ provides a $\mathbb{P}^d$-bundle
structure on a large open subset of $X$. Then  we recall some
results about degeneration of $\bbP^d$ in Section
\ref{section:degeneration-P^d} and prove  in Section
\ref{section:bundle structure} that the
$\bbP^d$-bundle structure holds on entire $X$.
 Finally, in Section
\ref{section:hyperbolicity}, we  show that the base $T$ is
hyperbolic, and complete the proof of Theorem \ref{thm:main-theorem}, Theorem \ref{thm:simply-connected-Pn} and Corollary \ref{cor:existence-symmetric-forms}.
 \\

{\bf Acknowledgements.} We would like to thank Professors Ya Deng, St\'ephane Druel, Baohua Fu, Andreas H\"oring, Thomas Peternell and Zhiyu Tian for inspiring discussions and useful communications. We would also like to thank Professor
Shing-Tung Yau for his valuable help, support and guidance.   The first-named author is supported by China Postdoctoral Science Foundation (2019M650873).

\vskip 2\baselineskip

\section{Preliminaries}


\label{section:pre}

Throughout this paper, we work over $\mathbb{C}$, the field of complex numbers.  All manifolds and varieties   are supposed to be irreducible.  The  following statement is a refined version of the negativity lemma (see \cite[Lemma 3.39]{KollarMori1998}).

\begin{lemma}
    \label{lemma:negativity-lemma-line-bundle}
    Let $f\colon X\to Y$ be a projective birational morphism between normal varieties of dimension $n\geqslant 2$.
    Assume that  $D$ is a $f$-exceptional $\mathbb{Q}$- Cartier $\mathbb{Q}$-divisor  which has at least one positive coefficient. Then  there is a family $\{C_\gamma\}_{\gamma\in \Gamma}$ of complete $f$-exceptional curves  such that $C_\gamma\cdot D<0$ for all $\gamma\in \Gamma$.

    Furthermore, for any fixed subvariety $W\subseteq X$ of codimension at least $2$, and for any fixed subvariety $V\subseteq X$ of codimension at least $3$,   a general member $C$ of $\{C_\gamma\}_{\gamma\in \Gamma}$ is not contained in $W$, and is disjoint from $V$.
\end{lemma}

\begin{proof}
    We first assume that $n=2$. Let $r\colon X'\to X$ be a desingularization and  $D'=r^*D$. Set $h=f\circ r$.
    Then as in the proof of \cite[Lemma 3.41]{KollarMori1998}, there is a component $C'$ of $D'$ with positive coefficient such that $C'\cdot D'<0$.
    Since $D'$ is relatively numerically trivial over $X'$, we see that $C'$ is not $r$-exceptional. Let $C=r(C')$.
    Then $C$ is a component of $D$ with positive coefficient such that $C\cdot D<0$.

    Next we study the general case. Let $D_1$ be a component of $D$ with positive coefficient and let  $d=\dim f(D_1)$. We chose $n-2$ hypersurfaces $H_1,\dots, H_{n-2}$ in $X$ such that $H_{i}$ is the pullback of some general very ample divisor in $Y$ for $i\le d$ and  is a general very ample divisor in $X$ for $i> d$. Let $S$  be the surface cut out by these hyperplanes.   Then $S$ is a normal surface. Let $T$ be the normalization of $f(S)$ and  denote by    $g \colon S\to T$ the natural morphism.

    Let $A$ be the cycle theoretic intersection of $D$ and $S$.  Since $D$ is $\mathbb{Q}$-Cartier, so is $A$. We note that $A$ has at least one positive coefficient.
      From the first paragraph, we see that $A$ is not $g$-nef, and there is a component $Z$ of $A$ with positive coefficient such that $Z\cdot A <0$. Therefore, $Z\cdot D<0$.    By  deforming  the hypersurfaces $H_i$, we can deform the curve $Z$ into  a family  $\{C_\gamma\}_{\gamma\in \Gamma}$.     The last assertion of the lemma follows from the fact that every curve  $C_\gamma$ is contained in the complete  intersection of $n-2$ base-point-free big divisors.
\end{proof}

\begin{lemma}
    \label{lemma:numerical-proportional-preseved}
    Let $f\colon X\rightarrow Y$ be a  morphism between normal projective varieties. Let $C_1$ and $C_2$ be two complete   curves in $Y$ and let $C_1'$ and $C_2'$ be two curves in $X$ such that $f(C_i')=C_i$ for each $i$. If $C'_1$ is numerically proportional to $C_2'$ in $X$, then $C_1$ is numerically proportional to $C_2$ in $Y$.
\end{lemma}

\begin{proof}
    By assumption, there exists a positive rational number $r\in \bbQ$ such that
    \begin{equation*}
    \label{equation:numerically-proportional}
    C_1'\equiv_{\bbQ} r C_2'.
    \end{equation*}
    Denote by $d_i$ the degree of the finite morphism $f\vert_{C_i'}\colon C_i'\rightarrow C_i$. Let $\sL$ be a line bundle  on $Y$. By projection formula, we have, for each $i$,
    \[
    c_1(f^*\sL)\cdot C_i'=d_i c_1(\sL)\cdot C_i.
    \]
    Combining with the first equation
    derives
    \[
    c_1(\sL)\cdot C_1=\frac{rd_2}{d_1}c_1(\sL)\cdot C_2.
    \]
    As $\sL$ is arbitrary, we conclude that $C_1$ is numerically proportional to $C_2$.
\end{proof}

\begin{lemma}
    \label{lemma:rational-section-birational-morphism}
    Let $f\colon X\rightarrow Y$ be a birational projective morphism  between normal quasi-projective varieties. Assume that $Y$ is smooth. Let $C$ be a complete curve in $Y$. Then there is a complete irreducible and reduced curve $C'$ in $X$ such that $f\vert_{C'}\colon C'\rightarrow C$ is a birational morphism.
\end{lemma}

\begin{proof}
    By \cite[II, Theorem 7.17]{Hartshorne1977}, there exists a coherent sheaf of ideals $\sI$ on $Y$ such that $X$ is isomorphic to the blowing-up of $X$ with respect to $\sI$. By Hironaka's resolution theorem, there exists a finite sequence $g\colon W\rightarrow Y$ of blowups with smooth center such that $g^*\sI$ is invertible. In particular, by the univeral property of blowing-up (see \cite[II, Proposition 7.14]{Hartshorne1977}), $g$ factors through $f\colon X\rightarrow Y$. Thus, by replacing $X$ by $W$, we may assume that $f$ is the composition of a sequence of blowups at smooth center. Using induction, we then reduce it to the case when $f$   is a blowup of smooth center $Z\subseteq Y$.

    If $C\not\subset Z$ then we can take $C'$ to be the strict transform of $C$ in $X$. Assume that $C\subseteq Z$. We note that $f$ is a projective bundle over $Z$. Hence if $\overline{C}$ is   the normalization of $C$, then $\overline{C}\times_Y X$ is a projective bundle  on $\overline{C}$.
    Such a projective bundle admits a   section, whose image is $D$. Let $C'$ be the image of $D$ in $X$. Then $f\vert_{C'}$ is a birational morphism onto $C$.
\end{proof}

\vskip 2\baselineskip

\section{Positivity of $\bbQ$-twisted coherent sheaves}


\label{section:Q-twisted}

 For readers' convenience, we collect some basic properties on positivity of $\bbQ$-twisted coherent sheaves.

\subsection{Nefness of $\bbQ$-twisted coherent sheaves}
Let $\sF$ be a coherent sheaf on a  variety $X$.  $\sF$ is called a vector bundle if it is locally free.
The singular locus $\mathrm{Sing}(\sF)$ of $\sF$ is the smallest closed subset of $X$ such that $\sF$ is locally free over $X\backslash \mathrm{Sing}(\sF)$.  We denote by $\rk(\sF)$ the rank of $\sF$.
The dual sheaf $\sHom{\sF}{\sO_X}$ is denoted by $\sF^*$ and the reflexive hull of $\sF$ is the double dual $\sF^{**}$.
Given a morphism $\gamma\colon Y\rightarrow X$, we denote by $\gamma^{[*]}\sF\coloneqq (\gamma^*\sF)^{**}$ the reflexive pullback.  The $m$-th reflexive symmetric power and $q$-th reflexive exterior product of $\sF$ are
\begin{center}
    $S^{[m]}\sF\coloneqq(\Sym^m\sF)^{**}$ and $\wedge^{[q]}\sF\coloneqq (\wedge^q\sF)^{**}$.
\end{center}
The determinant $\det(\sF)$ of $\sF$ is defined as $\wedge^{[r]}\sF$, where $r=\rk(\sF)$. We denote by $N^1(X)_{\bbQ}$ the finite-dimensional $\bbQ$-vector space of numerical equivalence classes of $\bbQ$-Cartier divisors on $X$. If we assume moreover that a coherent sheaf $\sF$ is locally free in codimension $1$ and that $\det(\sF)$ is $\bbQ$-Cartier, then   the first Chern class $c_1(\sF)$ of $\sF$ is defined as
\[
c_1(\sF)\defeq c_1(\det(\sF))\in N^1(X)_{\bbQ}
\]
and the averaged first Chern class $\mu(\sF)$ of $\sF$ is given  by
\[
\mu(\sF) \defeq \frac{1}{r}c_1(\sF)\in N^1(X)_{\bbQ}.
\]

\noindent
The projectivization $\bbP(\sF)$ is defined by $\bbP(\sF)\coloneqq \Proj{Sym^{\bullet}\sF}$.  If $\sF$ is a vector bundle, then $\mathbb{P}(\sF)$ is the projective bundle of hyperplanes in $\sF$.    We denote by $\sO_{\bbP(\sF)}(1)$ the tautological line bundle on $\bbP(\sF)$  and   by $\zeta_{\sF}$ the tautological class
\[c_1(\sO_{\bbP(\sF)}(1))\in N^1(\bbP(\sF)).\]

Now we introduce formally  $\bbQ$-twisted coherent sheaves, which extends the notion of $\bbQ$-twisted vector bundles (see \cite[\S\,6.2]{Lazarsfeld2004a}).

\begin{defn}
    \begin{enumerate}
        \item A $\bbQ$-twisted coherent sheaf $\sF\hspace{-0.8ex}<\hspace{-0.8ex}\delta\hspace{-0.8ex}>$ on a projective variety $X$ is an ordered pair consisting of a coherent sheaf $\sF$ on $X$,
        and a numerical equivalence class $\delta\in N^1(X)_{\bbQ}$.

        \item A $\bbQ$-twisted coherent sheaf $\sF\hspace{-0.8ex}<\hspace{-0.8ex}\delta\hspace{-0.8ex}>$ on a projective variety $X$ is called nef (resp. strictly nef, ample) if
        \[\zeta_{\sF}+p^*\delta\in N^1(\bbP(\sF))_{\bbQ}\]
        is a nef (resp. strictly nef, ample) $\bbQ$-Cartier divisor class where $p\colon \mathbb{P}(\sF) \to  X$ is the projection.
    \end{enumerate}
\end{defn}

The following result is  a Barton-Kleiman type criterion for
$\bbQ$-twisted coherent sheaves (see also \cite[Proposition
6.1.18]{Lazarsfeld2004a} and \cite[Proposition 2.1]{LiOuYang2019}).

\begin{prop}\label{BK criterion}
    Let $\sF\hspace{-0.8ex}<\hspace{-0.8ex}\delta\hspace{-0.8ex}>$ be a $\bbQ$-twisted coherent sheaf over a projective variety $X$. Then $\sF\hspace{-0.8ex}<\hspace{-0.8ex}\delta\hspace{-0.8ex}>$ is nef   if and only if given any finite morphism $\nu\colon C\rightarrow X$ from a smooth complete curve $C$ to $X$, and given any quotient line bundle $\sL$ of $\nu^*\sF$, one has
    \begin{equation}
        \det(\sL)+\deg(\nu^*\delta)\geqslant 0\label{BK}
    \end{equation}
    In particular, $\sF\hspace{-0.8ex}<\hspace{-0.8ex}\delta\hspace{-0.8ex}>$ is nef   if and only if the restriction $\sF\hspace{-0.8ex}<\hspace{-0.8ex}\delta\hspace{-0.8ex}>\vert_C$ is nef  for any complete curve $C\subseteq X$. The same criterion holds for strict nefness if  the  inequality in (\ref{BK}) is replaced by the strict  one.
\end{prop}

\begin{proof}
    Let $p\colon \mathbb{P}(\sF) \to X$ be the natural projection.
    Since $\zeta_\sF+p^*\delta$ is $p$-relatively ample,  we only need to consider complete curves in $\mathbb{P}(\sF)$ which are not contracted by $p$.
    Let $B$ be such a curve and $C$ its normalization. We denote by $\nu \colon C\to X$ the morphism induced by the projection from $\bbP(\sF)\rightarrow X$. Then $\nu$ is a finite morphism.
  By \cite[(4.1.3) and Proposition 4.2.3]{Grothendieck1961},  quotient line bundles $\nu^*\sF\rightarrow \sL$ correspond one-to-one  to  sections $\varphi\colon C\rightarrow \bbP(\nu^*\sF)$ with the following commutative diagram
     \[
    \begin{tikzcd}[column sep=large, row sep=large]
        \bbP(\nu^*\sF)\ar[r,"{\pi}"] \ar[d,"{p'}"]  & \bbP(\sF)\ar[d,"{p}"]\\
        C\ar[r,"{\nu}"] \arrow[bend left]{u}{\varphi}    &  X
    \end{tikzcd}
    \]
    such that
        $\bbP(\nu^*\sF)=\bbP(\sF)\times_{X} C$, $\sO_{\bbP(\nu^*\sF)}(1)=\pi^*\sO_{\bbP(\sF)}(1)$
    and
    \[\sL\cong  \varphi^*\sO_{\bbP(\nu^*\sF)}(1)=\varphi^*\pi^*\sO_{\bbP(\sF)}(1).\]
This implies $
        \deg(\sL)+\deg(\nu^*\delta) =  B\cdot (\zeta_\sF+p^*\delta),
    $
and we can conclude the criterion for nefness. The proof for strict nefness is similar.
\end{proof}

As an  application, one has the following useful criterion. 

\begin{cor}\label{Image-nef}
    Let $C$ be an irreducible projective curve and  $\sF\hspace{-0.8ex}<\hspace{-0.8ex}\delta\hspace{-0.8ex}>$  a nef (resp. strictly nef) $\bbQ$-twisted sheaf on $C$. If $\varphi\colon\sF\rightarrow \sQ$ is a morphism of coherent sheaves which is generically surjective, then the $\bbQ$-twisted sheaf $\sQ\hspace{-0.8ex}<\hspace{-0.8ex}\delta\hspace{-0.8ex}>$ is nef (resp. strictly nef).
\end{cor}

\begin{proof}
    Let $\nu\colon D\rightarrow C$ be a finite morphism from a smooth irreducible curve $D$ to $C$ and let $\nu^*\sQ\rightarrow \sL$ be a quotient line bundle.
    Since $\varphi$ is generically surjective, the composition
    $
    \nu^*\sF\rightarrow \nu^*\sQ\rightarrow \sL
    $
    is non-zero. Denote its image by $\sL'$. Since $\sL$ is torsion free, so is  $\sL'$. Thus $\sL'$ is a line bundle on the smooth curve $D$. If $\sF\hspace{-0.8ex}<\hspace{-0.8ex}\delta\hspace{-0.8ex}>$ is nef (resp. strictly nef),  then it follows from Proposition \ref{BK criterion} that
    \begin{center}
        $\deg(\sL')+\deg(\nu^*\delta)\geqslant 0$ (resp. $>0$).
    \end{center}
    We note that $\deg(\sL)\geqslant \deg(\sL')$. By applying Proposition \ref{BK criterion} again, we conclude that $\sQ\hspace{-0.8ex}<\hspace{-0.8ex}\delta\hspace{-0.8ex}>$ is nef (resp. strictly nef).
\end{proof}

The following proposition shows that nef $\bbQ$-twisted coherent sheaves are limits of ample $\bbQ$-twisted coherent sheaves.

\begin{prop}\cite[Proposition 6.2.11]{Lazarsfeld2004a}\label{Limit-ample}
    Let $\sF\hspace{-0.8ex}<\hspace{-0.8ex}\delta\hspace{-0.8ex}>$ be a $\bbQ$-twisted coherent sheaf over a projective variety $X$. Then $\sF\hspace{-0.8ex}<\hspace{-0.8ex}\delta\hspace{-0.8ex}>$ is nef if and only if $\sF\hspace{-0.8ex}<\hspace{-0.8ex}\delta+h\hspace{-0.8ex}>$ is ample for any ample class $h\in N^1(X)_{\bbQ}$.
\end{prop}

\begin{proof}
 By definition,   if $\sF\hspace{-0.8ex}<\hspace{-0.8ex}\delta+h\hspace{-0.8ex}>$ is ample for every ample class $h$, then  $\zeta_{\sF}+p^*(\delta+h)$ is ample. Hence  $\zeta_{\sF}+p^*\delta$   is nef.
  Assume conversely that $\sF\hspace{-0.8ex}<\hspace{-0.8ex}\delta\hspace{-0.8ex}>$ is nef. Since  $\zeta_{\sF}+p^*\delta$ is $p$-ample, by \cite[Proposition 4.4.10]{Grothendieck1961},  $\varepsilon(\zeta_{\sF}+p^*\delta)+p^*h$ is ample for  $0<\varepsilon \ll 1$. As $\zeta_F+p^*\delta$ is nef, it follows that
    \[\zeta_{\sF}+p^*\delta+p^*h=(1-\varepsilon)(\zeta_{\sF}+p^*\delta)+\varepsilon(\zeta_{\sF}+p^*\delta)+p^*h\]
    is ample by \cite[Corollary 1.4.10]{Lazarsfeld2004}.
\end{proof}

\noindent Similarly, one has
\begin{cor}\label{Symmetry-wedge}
    Let $\sE\hspace{-0.8ex}<\hspace{-0.8ex}\delta\hspace{-0.8ex}>$ and $\sF\hspace{-0.8ex}<\hspace{-0.8ex}\delta\hspace{-0.8ex}>$ be two nef $\bbQ$-twisted sheaves on a  projective variety $X$. Then the tensor product $(\sE\otimes \sF)\hspace{-0.8ex}<\hspace{-0.8ex}2\delta\hspace{-0.8ex}>$ is nef. In particular, $S^m\sE\hspace{-0.8ex}<\hspace{-0.8ex}m\delta\hspace{-0.8ex}>$ and $\wedge^q\sE\hspace{-0.8ex}<\hspace{-0.8ex}q\delta\hspace{-0.8ex}>$ are nef for all $m\geqslant 1$ and $1\leqslant q\leqslant \rk(\sE)$.
\end{cor}

\subsection{Almost nef $\bbQ$-twisted coherent sheaves}

For a  strictly nef subsheaf of the tangent bundle, its saturation is \textit{a priori} not nef.   We therefore consider a weaker positivity. The following notion of almost nef line bundle was introduced in \cite{DemaillyPeternellSchneider2001}. We extend this to the setting of $\mathbb{Q}$-twisted sheaves.

\begin{defn}
    Let $\sF$ be a coherent sheaf on a  projective variety $X$, and let $\delta\in N^1(X)_{\bbQ}$ be a $\bbQ$-Cartier divisor class.
    The $\bbQ$-twisted sheaf $\sF\hspace{-0.8ex}<\hspace{-0.8ex}\delta\hspace{-0.8ex}>$ is said to be almost nef, if there is a countable family $(Z_i)_{i\in \mathbb{N}}$ of proper subvarieties of $X$ such that $\sF\hspace{-0.8ex}<\hspace{-0.8ex}\delta\hspace{-0.8ex}>\vert_C$ is nef for all irreducible curves $C\not\subset \cup_{i\in \mathrm{N}} Z_i$.

     If $\sF\hspace{-0.8ex}<\hspace{-0.8ex}\delta\hspace{-0.8ex}>$ is almost nef,  its negative locus  $\bbS(\sF\hspace{-0.8ex}<\hspace{-0.8ex}\delta\hspace{-0.8ex}>)$  is the smallest countable union of  closed subvarieties such that $\sF\hspace{-0.8ex}<\hspace{-0.8ex}\delta\hspace{-0.8ex}>\vert_C$ is nef for all irreducible curves $C\not\subset \bbS(\sF\hspace{-0.8ex}<\hspace{-0.8ex}\delta\hspace{-0.8ex}>)$.
\end{defn}

\begin{remark}
        By \cite[Proposition 3.3]{DemaillyPeternellSchneider2001} and \cite[Theorem 0.2]{BoucksomDemaillyPuaunPeternell2013}, a $\bbQ$-Cartier divisor $D$ on a projective manifold $X$ is almost nef if and only if $D$ is pseudo-effective.
        However, we remark that in general the negative locus $\bbS(D)$ of $D$ is   a proper subset of the non-nef locus $\bbB_{-}(D)$ of $D$ (e.g. \cite[Remark 6.3]{BoucksomDemaillyPuaunPeternell2013}).
\end{remark}

We collect some basic properties of $\bbQ$-twisted almost nef coherent sheaves.

\begin{prop}\label{Almost-nef-properties}
    Let $\delta\in N^1(X)_{\bbQ}$ be a $\bbQ$-Cartier divisor class on a projective variety $X$, and let $\sE$, $\sF$ and $\sG$ be coherent sheaves on $X$.
    \begin{enumerate}
        \item If   $\sE\hspace{-0.8ex}<\hspace{-0.8ex}\delta\hspace{-0.8ex}>$ is almost nef and if $\sigma\colon \sE\rightarrow \sQ$ generically surjective, then $\sQ\hspace{-0.8ex}<\hspace{-0.8ex}\delta\hspace{-0.8ex}>$ is almost nef. Moreover, $\bbS(\sQ\hspace{-0.8ex}<\hspace{-0.8ex}\delta\hspace{-0.8ex}>)$ is contained in the union of $\bbS(\sE\hspace{-0.8ex}<\hspace{-0.8ex}\delta\hspace{-0.8ex}>)$ and the support of the torsion sheaf $\sG/\sigma(\sG)$.

        \item If $\sE\hspace{-0.8ex}<\hspace{-0.8ex}\delta\hspace{-0.8ex}>$ is almost nef, then $S^{[m]}\sE\hspace{-0.8ex}<\hspace{-0.8ex}m\delta\hspace{-0.8ex}>$ and $\wedge^{[q]}\sE\hspace{-0.8ex}<\hspace{-0.8ex}q\delta\hspace{-0.8ex}>$ are almost nef for all $m,q>0$.  Their negative loci are contained in the union of  $\mathrm{Sing}(\sE)$ and $\bbS(\sE\hspace{-0.8ex}<\hspace{-0.8ex}\delta\hspace{-0.8ex}>)$.

        \item  Let $p \colon \mathbb{P}(\sE) \to X$ be the natural projection. If $\zeta_{\sE}+p^*\delta$ is almost nef  and its negative locus $\bbS(\zeta_{\sE}+p^*\delta)$ does not dominate $X$, then  $\sE\hspace{-0.8ex}<\hspace{-0.8ex}\delta\hspace{-0.8ex}>$   is almost nef and $\bbS(\sE\hspace{-0.8ex}<\hspace{-0.8ex}\delta\hspace{-0.8ex}>)$ is contained in   $p(\bbS(\zeta_{\sE}+p^*\delta))$.

        \item Let $0\rightarrow \sF\rightarrow \sE\rightarrow \sQ\rightarrow 0$ be an exact sequence of coherent sheaves. If both $\sF\hspace{-0.8ex}<\hspace{-0.8ex}\delta\hspace{-0.8ex}>$ and $\sQ\hspace{-0.8ex}<\hspace{-0.8ex}\delta\hspace{-0.8ex}>$ are almost nef, then $\sE\hspace{-0.8ex}<\hspace{-0.8ex}\delta\hspace{-0.8ex}>$ is almost nef and $\bbS(\sE\hspace{-0.8ex}<\hspace{-0.8ex}\delta\hspace{-0.8ex}>)$ is contained in $\bbS(\sF\hspace{-0.8ex}<\hspace{-0.8ex}\delta\hspace{-0.8ex}>)\cup\bbS(\sQ\hspace{-0.8ex}<\hspace{-0.8ex}\delta\hspace{-0.8ex}>)$.
    \end{enumerate}
\end{prop}

\begin{proof}
    For (1), let $C\subseteq X$ be a complete curve such that $C\not\subset \bbS(\sE\hspace{-0.8ex}<\hspace{-0.8ex}\delta\hspace{-0.8ex}>)\cup \supp(\sG/\sigma(\sG))$. Then the induced morphism $\sE\vert_C\rightarrow \sG\vert_C$ is generically surjective and we conclude by Corollary \ref{Image-nef} that $\sG\hspace{-0.8ex}<\hspace{-0.8ex}\delta\hspace{-0.8ex}>\vert_C$ is nef.
    For (2), we only prove the statement for $S^{[m]}\sE\hspace{-0.8ex}<\hspace{-0.8ex}m\delta\hspace{-0.8ex}>$, and the case of exterior power is similar. Let $C\subseteq X$ be a complete  curve such that $C\not\subset \mathrm{Sing}(\sE)\cup \bbS(\sE\hspace{-0.8ex}<\hspace{-0.8ex}\delta\hspace{-0.8ex}>)$. Then  the induced morphism $(S^m\sE)\vert_C\rightarrow (S^{[m]}\sE)\vert_C$ is generically surjective. Since   symmetric powers commute with pullbacks (see for instance \cite[II, Exercise 5.16]{Hartshorne1977}), the following morphism
    \[S^m(\sE\vert_C)=(S^m\sE)\vert_C\rightarrow (S^{[m]}\sE)\vert_C \]
    is generically surjective. As $\sE\hspace{-0.8ex}<\hspace{-0.8ex}\delta\hspace{-0.8ex}>\vert_C$ is nef by assumption, we obtain by Corollary \ref{Symmetry-wedge} and Corollary \ref{Image-nef} that $(S^{[m]}\sE\hspace{-0.8ex}<\hspace{-0.8ex}m\delta\hspace{-0.8ex}>)\vert_C$ is nef.
    For (3), let $C\subseteq X$ be a complete curve  not contained in  $p(\bbS(\zeta_{\sE}+p^*\delta))$. Then we have the following commutative diagram
    \[\begin{tikzcd}[column sep=large, row sep=large]
        \bbP(\sE|_C)\dar[swap]{p'}\ar[r,"{j}"] & \bbP(\sE)\ar[d,"{p}"]\\
        C\ar[r,"{i}"]                        & X
      \end{tikzcd}
    \]
    such that
    \[\zeta_{\sE|_C}+p'^*(\delta\vert_C)=j^*(\zeta_{\sE}+p^*\delta).\]
    This implies that $\zeta_{\sE|_C}+p'^*(\delta\vert_C)$ is nef, that is, $\sE\hspace{-0.8ex}<\hspace{-0.8ex}\delta\hspace{-0.8ex}>\vert_C$ is nef.
    For (4), let $\nu \colon C\to X$ be finite  morphism from a smooth complete curve  $C$ to $X$ such that $\nu(C)$ is  not contained in  $ \bbS(\sF\hspace{-0.8ex}<\hspace{-0.8ex}\delta\hspace{-0.8ex}>)\cup \bbS(\sG\hspace{-0.8ex}<\hspace{-0.8ex}\delta\hspace{-0.8ex}>)$.
    Then we have an exact sequence
    \[\nu^*\sF \rightarrow \nu^*\sE\rightarrow \nu^*\sG\rightarrow 0.\]
    Let $\nu^*\sE\rightarrow \sL$ be a quotient line bundle. Then either the composition
    $\nu^*\sF\rightarrow \nu^*\sE\rightarrow \sL$
    is not zero, or there exists a factorization
    $\nu^*\sE\rightarrow \nu^*\sG\rightarrow \sL.$
    Since both $\sF\hspace{-0.8ex}<\hspace{-0.8ex}\delta\hspace{-0.8ex}>\vert_C$ and $\sG\hspace{-0.8ex}<\hspace{-0.8ex}\delta\hspace{-0.8ex}>\vert_C$ are nef, by  Corollary \ref{Image-nef},   we  obtain that $\sL\hspace{-0.8ex}<\hspace{-0.8ex}\delta\hspace{-0.8ex}>\vert_C$ is nef. Hence, $\sE\hspace{-0.8ex}<\hspace{-0.8ex}\delta\hspace{-0.8ex}>\vert_C$ is nef.
\end{proof}

\vskip 2\baselineskip

\section{Projectively flat vector bundles}

\label{section:proj-flat}
The notion of numerically flat vector bundle was firstly introduced in \cite[Definition 1.7]{DemaillyPeternellSchneider1994}.
We can extend this to $\mathbb{Q}$-twisted vector bundles  as follows: a $\mathbb{Q}$-twisted vector bundle $\sE\hspace{-0.8ex}<\hspace{-0.8ex}\delta\hspace{-0.8ex}>$ is numerically flat if it is nef and it has trivial first Chern class. We also introduce the following definition.

\begin{defn}
    Let $\sE$ be a vector bundle on a projective manifold $X$ with projectivization $p\colon \bbP(\sE)\rightarrow X$. The normalized tautological class $\Lambda_{\sE}$ is defined as $\zeta_{\sE}-p^*\mu(\sE)\in N^1(\bbP(\sE))_{\bbQ}$.  $\sE$ is called \emph{numerically projectively flat} if $\Lambda_{\sE}$ is nef.
\end{defn}

\noindent Equivalently, a vector bundle $\sE$ is numerically projectively flat if and only if the $\mathbb{Q}$-twisted vector bundle  $\sE\hspace{-0.8ex}<\hspace{-0.8ex}-\mu(\sE)\hspace{-0.8ex}>$ is numerically flat.

\subsection{Characterization of numerically projectively flat vector bundles}

Recall that a $C^\infty$ complex vector bundle $\sE$ is called \emph{projectively flat} if there exists an affine connection $\nabla$ such that its curvature $\nabla^2=\alpha\cdot \id_{\sE}$ for some complex  $2$-form $\alpha$. A holomorphic vector bundle $\sE$ is \emph{projectively Hermitian flat} if it admits a smooth Hermitian metric $h$ such that its Chern  curvature tensor $R=\nabla^2$ can be written as  $R=\alpha \cdot\id_{\sE}$ for some $2$-form $\alpha$.  In particular, the associated projectivized bundle $\bbP(\sE)$ is given by a representation $\pi_1(X)\rightarrow PU(r)$ (see \cite[Corollary 4.3]{Nakayama2004}). The following theorem is derived from the study of stable vector bundles and Einstein-Hermitian metrics by Narasimhan-Seshadri \cite{NarasimhanSeshadri1965}, Mehta-Ramanathan \cite{MehtaRamanathan1981/82,MehtaRamanathan1984}, Donaldson \cite{Donaldson1985}, Uhlenbeck-Yau \cite{UhlenbeckYau1986}, and Bando-Siu \cite{BandoSiu1994}. One can find a complete proof in \cite[IV, Theorem 4.1]{Nakayama2004}.

\begin{thm}
\label{thm:num-proj-flat-criterion-Nakayama}
Let $\sE$ be a reflexive sheaf of rank $r$ on a projective manifold $X$ of dimension $n$. Then the following assertions are equivalent.
\begin{enumerate}
\item $\sE$ is a numerically projectively flat vector bundle;
\item $\sE$ is semistable with respect to some ample divisor $A$ and the following equality holds: \[\left(c_2(\sE)-\frac{r-1}{2r}c_1^2(\sE)\right)\cdot A^{n-2}=0;\]
\item $\sE$ is a vector bundle and there exists a filtration of  subbundles
\[\{0\}=\sE_0\subsetneq \sE_1\subsetneq \dots\subsetneq \sE_{p-1}\subsetneq \sE_p=\sE\]
such that $\sE_i/\sE_{i+1}$ are  projectively Hermitian  flat and that the averaged first Chern classes $\mu(\sE_i/\sE_{i-1})$ are all equal to $\mu(\sE)$.
\end{enumerate}
\end{thm}

One can easily derive the following lemma from definition and Theorem \ref{thm:num-proj-flat-criterion-Nakayama}.

\begin{lemma}
    \label{lemma:prop-num-proj-flat}
    Let $\sE$ be a numerically projectively flat vector bundle on a projective manifold $Y$.
    \begin{enumerate}
        \item If $\det(\sE)$ is nef, then $\sE$ is nef.

        \item If $\sL$ is a line bundle on $Y$, then $\sE\otimes \sL$ is numerically projectively flat.

        \item If $f\colon X\rightarrow Y$ is a morphism from a projective manifold $X$ to $Y$, then $f^*\sE$ is numerically projectively flat.

        \item $\sE^*$ is numerically projectively flat.
    \end{enumerate}
\end{lemma}

In  \cite{HoeringPeternell2019}, H\"oring and Peternell characterized numerically flat vector bundles by using almost nefness, instead of nefness in the original definition of \cite{DemaillyPeternellSchneider1994}. Here we quote their theorem in a special case and  refer the readers to \cite{HoeringPeternell2019} for the complete statement.

\begin{thm}\cite[Theorem 1.8]{HoeringPeternell2019}\label{HP-Num-flatness}
    Let $X$ be an $n$-dimensional projective manifold. Let $\sF$ be an almost nef reflexive coherent sheaf on $X$ such that $c_1(\sF)\cdot A^{n-1}=0$ for some ample divisor $A$ on $X$. Then $\sF$ is locally free and numerically flat.
\end{thm}

\begin{proof}
    According to \cite[Theorem 1.8]{HoeringPeternell2019}, there exists a finite cover $\gamma\colon \widetilde{X}\rightarrow X$, \'etale in codimension one, such that the reflexive pullback $\gamma^{[*]}\sF$ is locally free and numerically flat. Since $X$ is smooth, $\gamma$ is actually \'etale and $\gamma^*\sF=\gamma^{[*]}\sF$. Since $\gamma$ is \'etale, this implies that $\sF$ itself is locally free and numerically flat.
\end{proof}

\noindent Before giving the  proof of Theorem \ref{Num-Projectily-flatness-criterion},  we need the following elementary lemma.

\begin{lemma}\label{Chern-class-symmetric-power}
    Let $X$ be a projective manifold of dimension $n\geqslant 2$, and let $\sF$ be a reflexive   sheaf of rank $r\geqslant 2$ on $X$. For any positive integer $m\geqslant 2$, we have
    \begin{equation}\label{Chern class}
    c_2(S^{[m]}\sF)=Ac_1^2(\sF)+Bc_2(\sF),
    \end{equation}
    where $A$ and $B$ are non-zero rational numbers  depending only on $m$ and $r$, and satisfy
    \begin{equation}\label{AB}
    A+\frac{r-1}{2r}B-\frac{(R-1)Rm^2}{2r^2}=0,
    \end{equation}
    where $R=\binom{r+m-1}{r}$ is the rank of $S^{[m]}\sF$.
\end{lemma}

\begin{proof}
    The existence of the expression \eqref{Chern class} is clear and the splitting principle asserts that $A$ and $B$ depend only on $m$ and $r$. To prove \eqref{AB}, it suffices to prove it for some special $\sF$ by the universal property of $A$ and $B$.

    Firstly we choose an ample line bundle $\sL$ on $X$ and let $\sF=\sL^{\oplus r}$. Then   $S^m\sF\otimes\sL^{*\otimes m}$ is a trivial  vector bundle. In particular, we have
    \[
    c_2(S^m\sF\otimes \sL^{*\otimes m})=0.
    \]
    By the formula of second Chern class of tensor products, we obtain that
    \begin{align*}
    c_2(S^m\sF)+(R-1)c_1(S^m\sF)\cdot (-m)c_1(\sL)+\binom{R}{2}m^2c_1^2(\sL)=0.
    \end{align*}
    On the other hand, since $\sF$ is numerically projectively flat,  we  have
    \[\left(c_2(\sF)-\frac{r-1}{2r}c_1^2(\sF)\right)\cdot A^{n-2}=0\]
    for any ample divisor $A$.  For $c_1(\sF)=rc_1(\sL)$ and $c_1(S^m\sF)=\frac{Rm}{r}c_1(\sF)$, one has
    \[\left(A+\frac{r-1}{2r}B+\frac{(R-1)Rm^2}{2r^2}\right)c_1^2(\sF)\cdot A^{n-2}=0.\]
    Since $c_1(\sF)$ is ample, we must have $c_1^2(\sF)\cdot A^{n-2}\not=0$. This shows that \eqref{AB} holds. To see that $A$ and $B$ are non-zero, we may consider the vector bundles
    \begin{center}
        $\sF'\cong  \sO_X^{\oplus (r-1)}\oplus \sL$ and $\sF''\cong
        \sO_{X}^{\oplus (r-2)}\otimes \sL^*\oplus\sL$.
    \end{center}
    Then $c_1(\sF')\not=0$, $c_2(\sF')\not=0$. A straightforward computation shows that   $c_2(S^m(\sF'))\not=0$ for $m\geqslant 2$. In particular, $A$ is non-zero. Similarly,  one can show $B$ is non-zero.
\end{proof}

\begin{proof}[Proof of Theorem \ref{Num-Projectily-flatness-criterion}]

    Suppose $C$ is a curve cut out by general elements in $\vert kA\vert$ for $k\gg 0$. Then $C$ is disjoint from   $\mathrm{Sing}(\sF)$.
    In particular, $\sF$ is locally free along $C$ and $\sF\hspace{-0.8ex}<\hspace{-0.8ex}\delta\hspace{-0.8ex}>\vert_C$ is a $\bbQ$-twisted nef vector bundle.
    Moreover,  we have $c_1(\sF\hspace{-0.8ex}<\hspace{-0.8ex}\delta\hspace{-0.8ex}>\vert_C)=0$ as $(c_1(\sF)+r\delta)\cdot C=0$. This implies that $\sF\vert_C$ is semistable. By Mehta-Ramanathan  theorem, $\sF$  is $A$-semistable.

    Let $m$ be a positive integer such that $m\delta$ is Cartier. Let $\sL$ be a line bundle such that $c_1(\sL)=m\delta$.
    Since $\sF\hspace{-0.8ex}<\hspace{-0.8ex}\delta\hspace{-0.8ex}>$ is almost nef, so is   $\sE\coloneqq S^{[m]}\sF\otimes \sL$ by Proposition \ref{Almost-nef-properties}. On the other hand, we know
    \[c_1(\sE)\cdot A^{n-1}=\left(\frac{Rm}{r}c_1(\sF)+Rc_1(\sL)\right)\cdot A^{n-1}=\frac{Rm}{r}(c_1(\sF)+r\delta)\cdot A^{n-1}=0,\]
    where $R$ is the rank of $S^{[m]}\sF$. Then we have $c_1(\sE)=0$ since $\det(\sE)$ is almost nef (see \cite[Lemma 6.5]{Peternell1994}).
    In particular, $\delta=-\mu(\sF)$.

     Moreover, by Theorem \ref{HP-Num-flatness}, $\sE$ is locally free and numerically flat. In particular, we have $c_2(\sE)\cdot H^{n-2}=0$. As $c_1(\sF)=-r\delta$, we have
    \begin{align*}
    c_2(\sE) & =c_2(S^{[m]}\sF)+(R-1)c_1(S^{[m]}\sF)\cdot c_1(\sL)+\frac{R(R-1)}{2}c_1^2(\sL)\\
    & = c_2(S^{[m]}\sF)-\frac{(R-1)Rm^2}{r^2}c^2_1(\sF)+\frac{(R-1)Rm^2}{2r^2}c_1^2(\sF)\\
    & = c_2(S^{[m]}\sF)- \frac{(R-1)Rm^2}{2r^2}c_1^2(\sF).
    \end{align*}
    Then  Lemma \ref{Chern-class-symmetric-power} yields
    \[\left(c_2(\sF)-\frac{r-1}{2r}c_1^2(\sF)\right)\cdot A^{n-2}=0.\]
    Thanks to Theorem \ref{thm:num-proj-flat-criterion-Nakayama}, we conclude that $\sF$ is a numerically projectively flat vector bundle.
\end{proof}

\subsection{Projectively flat connections}

In this subsection,  we prove Theorem \ref{thm:num-proj-flat=proj-flat} and it can be deduced from the following theorem.

\begin{thm}\label{thm:Kahler:proj-flat-connection}
Let $(X,\omega)$ be a compact K\"ahler manifold and $\sE$  a holomorphic vector bundle on $X$. Assume that there exists a filtration of  subbundles
\[\{0\}=\sE_0\subsetneq \sE_1\subsetneq \dots\subsetneq \sE_{p-1}\subsetneq \sE_p=\sE\]
such that $\sE_i/\sE_{i+1}$ are  projectively Hermitian  flat and that the averaged first Chern classes $\mu(\sE_i/\sE_{i-1})$ are all equal to $\mu(\sE)$.
Then $\sE$ is isomorphic to a  projectively flat holomorphic vector bundle $\sF$, i.e.  there exists a projectively flat connection $\nabla$
    on $\sF$ such that  $\nabla^{0,1}=\bar\partial_{\sF}$ where
    $\nabla^{0,1}$ is the $(0,1)$-part of $\nabla$.
\end{thm}

\noindent  It is well-known that  an affine connection $\nabla$ on a $C^\infty$ complex vector bundle $\sQ$ defines a holomorphic vector bundle structure on $\sQ$ if $(\nabla^{0,1})^2=0$ where $\nabla^{0,1}$ is the $(0,1)$ part of $\nabla$ (see e.g. \cite[Proposition 1.3.7]{Kobayashi2014}).  If in addition  $(\sQ, \nabla)$ is projectively flat,  then the projective bundle $\mathbb{P}(\sQ)$ is induced by a representation of the fundamental group $\pi_1(X)$ in $\mathrm{PGL}_{r}(\mathbb{C})$, where $r$ is the rank of $\sG$. The strategy of the proof of Theorem \ref{thm:Kahler:proj-flat-connection} is to construct a new  holomorphic structure on $\sQ$, where $\sQ$ is the underlying $C^\infty$ complex vector bundle of the holomorphic vector bundle $\sE$, and this new holomorphic structure is isomorphic to $\sE$.   We recall some elementary results.

\begin{lemma} \label{lemma:unique-proj-flat-connection}
    Let $(X,\omega)$ a compact K\"ahler manifold and let $\sE$ be  a
    $[\omega]$-stable vector bundle of rank $r$ on $X$. If $\sE$ satisfies
    \begin{equation}\label{BG}
    \int_X \left((r-1) c_1^2(\sE)-2r c_2(\sE)\right)
    \omega^{n-1}=0,
    \end{equation}
    then there exists a smooth Hermitian
    metric $h^\sE$ on $\sE$ with Chern connection $\nabla^\sE$ such that the Chern curvature satisfies
    \begin{equation*}
    R^\sE= (\nabla^\sE)^2 = \gamma \cdot h^\sE,
    \end{equation*}
    where $\gamma$ is the unique harmonic representative of the average first Chern class
    $\frac{1}{r}c_1(\sE)\in H_{\bp}^{1,1}(X,\mathbb{C})$ with respect to the K\"ahler metric
    $\omega$, i.e. $\bp\gamma=\bp^*\gamma=0$.
\end{lemma}

\begin{proof}
    Since $\sE$ is $[\omega]$-stable, by the
    Donaldson-Uhlenbeck-Yau theorem, there exists a smooth
    Hermitian-Einstein metric $h^\sE$ on $\sE$. That means $\Lambda_\omega
    R^\sE=c\cdot h^\sE$ for some constant $c$. By \cite[Theorem 4.4.7]{Kobayashi2014}, the equality (\ref{BG}) implies that
    $(\sE,h)$ is  projectively flat, that is $R^\sE=\frac{1}{r} \eta \cdot h^\sE$
    for some  smooth closed $(1,1)$-form $\eta$. By taking the trace of $R^\sE$ with
    respect to $h^\sE$, we have $$R^{\det \sE}=-\sqrt{-1} \partial \bp\log(\det
    h^\sE)=\mathrm{tr}_{h^\sE} R^\sE=\eta.$$ This shows that class of  $\eta$ is equal to the first Chern class of $\sE$.
    Finally, since $h^\sE$ is Hermitian-Einstein, we also have
    $\Lambda_\omega \eta=rc$.
    Hence  $\bp^*\eta=-\sqrt{-1}[\Lambda_\omega,\partial]\eta=0$. This implies that $\eta$ is  harmonic. The lemma then follows by setting $\gamma=\frac{1}{r}\eta.$
\end{proof}

\begin{cor}\label{DUY}
    Let $(X,\omega)$ a compact K\"ahler manifold. Suppose $\sE$ and $\sF$
    are $[\omega]$-stable vector bundles with
    \begin{equation}\label{average}
    \frac{c_1(\sE)}{\mathrm{rank}(\sE)}=\frac{c_1(\sF)}{\mathrm{rank}(\sF)}\in
    H_{\bp}^{1,1}(X,\mathbb{C})\cap H^2(X,\mathbb{Q}).
    \end{equation}
    If both $\sE$
    and $\sF$ satisfy the equality (\ref{BG}), then there exist Hermitian-Einstein
    metrics $h^\sE$ and $h^\sF$ on $\sE$ and $\sF$ respectively, such that
    \begin{equation*}
    R^\sE=\gamma \cdot h^\sE,\ \ \ R^\sF=\gamma \cdot h^\sF\
    \end{equation*}  where $\gamma$
    is the unique harmonic representative of the class of
    (\ref{average}).  In particular, the metric on $\sE^*\otimes \sF$ induced by $h^\sE$ and $h^\sF$ is Hermitian flat.
\end{cor}

The following lemma reveals the relationship between  isomorphism classes of holomorphic structures and Dolbeault cohomology groups.

\begin{lemma}
    \label{lemma:coh-class-iso-holomorphic-class}
    Let $\sF$ and $\sG$ be  two holomorphic  vector bundles on a complex manifold $X$.  We denote by $\sQ$ the underlying $C^\infty$ vector bundle of $\sG \oplus \sF.$
    Then there is a bijection between  the cohomology group $H^{0,1}_{\bp}(X, \sF^*\otimes \sG)$ and the   set  of isomorphism classes of holomorphic vector bundle structure $\sE$  on $\sQ$ which realizes $\sE$  as an extension  $0\to \sG \to \sE \to \sF \to 0$ of holomorphic vector bundles.
\end{lemma}

\begin{proof}
    To define such a holomorphic structure on $\sQ$, it is equivalent to give  a $(0,1)$-connection $\nabla^{0,1}$ on $\sQ$ of the form
    \begin{equation*}
    \nabla^{0,1}=\left[\begin{array}{cc} \bp_\sG&\eta\\
    0&\bp_\sF
    \end{array}\right]
    \end{equation*}
    such that $(\nabla^{0,1})^2=0$, where $\eta \in A^{0,1}(X,\sF^*\otimes \sG)$ is a smooth $(0,1)$-form.

    Now we assume that $\eta$ and $\eta'$ are two elements in $A^{0,1}(X,\sF^*\otimes \sG).$
    Then they induce  isomorphic extension structures, $\sE$ and $\sE'$,  on $\sQ$ if and only if there is some smooth form $\alpha \in A^0(X,\sF^*\otimes \sG)$ such that $\eta'=\eta+\bp \alpha$.
    Moreover, the corresponding isomorphism from $\sE$ to $\sE'$, as a smooth automorphism of $\sQ$,  is expressed as
    \begin{equation*}
    \left[\begin{array}{cc} 1_\sG& \alpha\\
    0&1_\sF
    \end{array}\right].
    \end{equation*}
    This finishes the proof.
\end{proof}

In the next lemma, we  prove Theorem \ref{thm:Kahler:proj-flat-connection} in a special case when $\sE$ is an extension of two projectively Hermitian flat vector bundles.

\begin{lemma}
    \label{bb} Let $(X,\omega)$ be a compact K\"ahler manifold
    and let
    $$
    0\to \sG_1\to \sE\to \sG_2\to 0
    $$
    be an exact sequence of holomorphic vector bundles on $X$.
    Suppose that $\sE$ is  $[\omega]$-semistable and satisfies the equality
    (\ref{BG}), and that $\sG_1$ and $\sG_2$ are  $[\omega]$-stable with the same average first Chern class.
    Let $\sQ$ be the underlying $C^\infty$ vector bundle of $\sG_1\oplus \sG_2$.
    Then there exists  a projectively flat
    connection $\nabla^\sF$ on $\sQ$ such that it defines a holomorphic structure $\sF$ isomorphic to $\sE$.
\end{lemma}

\begin{proof}
    It is easy to see that if $\sE$ is $[\omega]$-semistable and
    satisfies  (\ref{BG}), then both $\sG_1$ and $\sG_2$ satisfy (\ref{BG}).
    By  Corollary \ref{DUY}, there exist Hermitian-Einstein metrics $h^{\sG_1}$
    and $h^{\sG_2}$ on $\sG_1$ and $\sG_2$ respectively, such that their Chern curvatures  are given by
    \begin{equation*}\label{equal}
    R^{\sG_1}=\gamma \cdot h^{\sG_1},\ \ \
    R^{\sG_2}=\gamma \cdot h^{\sG_2},
    \end{equation*}
    where  $\gamma$ is the
    unique harmonic representative of the average first Chern class of $\sG_1$. Let
    $\nabla^{\sG_1}$ and $\nabla^{\sG_2}$ be the Chern connections on $(\sG_1,h^{\sG_1})$ and
    $(\sG_2,h^{\sG_2})$ respectively.
    Then the Chern curvature of
    $(\sG_2^*\otimes \sG_1, h^{\sG_2^*}\otimes h^\sG_1, \nabla^{\sG_2^*}\otimes \nabla^\sG_1)$  satisfies $R^{\sG_2^*\otimes \sG_1}=0$.
    The standard
    Bochner-Kodaira identity on $\sG_2^*\otimes \sG_1$ shows the following equalities on Laplacian operators on
    $A^{\bullet,\bullet}(X,\sG_2^*\otimes \sG_1)$,
    \begin{equation*}
    \Delta''=\Delta'+\sqrt{-1}[R^{\sG_2^*\otimes \sG_1},\Lambda_\omega]=\Delta'.
    \end{equation*}

\noindent  By partition of unity, there is a connection $\nabla^{\sE}$ on $\sQ$ with $(\nabla^\sE)^{0,1}=\bp_{\sE}$ of the form
    \begin{equation*}
    \nabla^\sE:=\left[\begin{array}{cc} \nabla^{\sG_1}&\beta\\
    0&\nabla^{\sG_2}
    \end{array}\right].
    \end{equation*}
    By using the  metrics $\omega$, $h^{\sG_2^*}$ and $h^{\sG_1}$, one has an isomorphism $ H^1(X,\sG_2^*\otimes \sG_1))\cong H_{\bp}^{0,1}(X,\sG_2^*\otimes \sG_1).$
    Let $\eta \in A^{0,1}(X,\sG_2^*\otimes \sG_1)$ be the unique harmonic representative of the class $[\beta]\in H_{\bp}^{0,1}(X,\sG_2^*\otimes
    \sG_1)$. Then $\Delta'' \eta =0$. Since $\Delta'=\Delta''$, we deduce that
    \begin{equation*}\label{harmonic}
    \nabla^{\sG_2^*\otimes
        \sG_1}\eta=0.
    \end{equation*}
    We define a connection $\nabla^\sF$ on $\sQ$ in the form
    \begin{equation*}
    \nabla^\sF:=\left[\begin{array}{cc} \nabla^{\sG_1}&\eta\\
    0&\nabla^{\sG_2}
    \end{array}\right].
    \end{equation*}
    Then by Lemma \ref{lemma:coh-class-iso-holomorphic-class}, the connection $\nabla^\sF$ defines a holomorphic structure $\sF$ on $\sQ$ isomorphic to $\sE$. Moreover, we have
    \begin{equation*}
    (\nabla^\sF)^2=\gamma\cdot (h^{\sG_1}\oplus h^{\sG_2}).
    \end{equation*} \label{flat}
    Indeed, this is equivalent to
   $ \nabla^{\sG_1}\circ \eta+\eta\circ \nabla^{\sG_2}=\nabla^{\sG_2^*\otimes
        \sG_1}\eta=0$
and it  follows by the choice of $\eta$. This completes the proof of the lemma.
\end{proof}

\noindent The next statement addresses a generalization of the construction  in Lemma \ref{bb}.

\begin{lemma}
    \label{lemma:connection-construction}
    Let $\{(\sG_i,h^{\sG_i}, \nabla_i)\}_{i=1}^p (p\geq 2)$  be a collection of projectively Hermitian flat vector bundles on a compact K\"ahler manifold $(X,\omega)$.  Let $\sQ$ be the underlying $C^\infty$ vector bundle of $\bigoplus_{i=1}^p \sG_i$. Assume that there is  a connection  on $\sQ$ of the form
    \[
    \nabla^{\sQ}_1 =
    \begin{bmatrix}
    \nabla_1 & \theta_{1,2} & \cdots & \cdots & \theta_{1,p-1} & \beta_{1,p} \\
    & \nabla_2  & \cdots & \cdots& \theta_{2,p-1} & \beta_{2,p} \\
    & &\ddots & \cdots & \vdots & \vdots \\
    & & & \ddots & \vdots & \vdots \\
    & & & & \nabla_{p-1} & \beta_{p-1,p} \\
    0 & & & & & \nabla_p
    \end{bmatrix}
    \]
    such that  for each $i<j$, the $(i,j)$ entry belongs to  $A^{0,1}(X, \sG_j^*\otimes \sG_i)$.
    Furthermore, we assume that  $(\nabla^{\sQ}_1)^{0,1})^2=0$ and that
    \[
    \begin{bmatrix}
    \nabla_1 & \theta_{1,2} & \cdots & \cdots & \theta_{1,p-1} \\
    &\nabla_2  & \cdots & \cdots& \theta_{2,p-1}  \\
    & & \ddots & \cdots & \vdots  \\
    & & & \ddots & \vdots  \\
    0 & & & & \nabla_{p-1}
    \end{bmatrix}^2
    =
    \begin{bmatrix}
    \gamma h^{\sG_1} &  0 & \cdots & \cdots & 0 \\
    & \gamma h^{\sG_2} & \cdots & \cdots& 0  \\
    & &\ddots & \cdots & \vdots  \\
    & & & \ddots & \vdots  \\
    0 & & & &  \gamma h^{\sG_{p-1}}
    \end{bmatrix},
    \]
    where $r_i$ is the rank of $\sG_i$ for $i=1,...,p.$

    Then  there is a connection on $\sQ$ of the form
    \[
    \nabla^{\sQ}_2 =
    \begin{bmatrix}
    \nabla_1 & \theta_{1,2} & \cdots & \cdots & \theta_{1,p-1} & \theta_{1,p} \\
    & \nabla_2  & \cdots & \cdots& \theta_{2,p-1} & \theta_{2,p} \\
    & &\ddots & \cdots & \vdots & \vdots \\
    & & & \ddots & \vdots & \vdots \\
    & & & & \nabla_{p-1} & \theta_{p-1,p} \\
    0 & & & & & \nabla_p
    \end{bmatrix}
    \]
    with $\theta_{i,p}\in A^{0,1}(X, \sG_p^*\otimes \sG_i)$ for $i<p$ such that its curvature is of the form
    $$
    (\nabla^{\sQ}_2)^2   = \gamma \cdot (h^{\sG_1}\oplus\cdots\oplus h^{\sG_p}).
$$
    Moreover, the holomorphic structures induced by $\nabla^\sQ_1$ and $\nabla^\sQ_2$ are isomorphic.
\end{lemma}

\begin{proof}
    We  prove it by induction on $p$. For $p=2$, it follows from Lemma \ref{bb}.  Assume that $p>2$ and that the lemma is true for smaller integers. By induction hypothesis, we can find $\theta_{2,p},...,\theta_{p-1,p}$ such that
    \[
    \begin{bmatrix}
    \nabla_2 & \theta_{2,3} & \cdots & \cdots & \theta_{2,p} \\
    &\nabla_3  & \cdots & \cdots& \theta_{3,p}  \\
    & & \ddots & \cdots & \vdots  \\
    & & & \ddots & \vdots  \\
    0 & & & & \nabla_{p}
    \end{bmatrix}^2
    =
    \begin{bmatrix}
    \gamma h^{\sG_2} &  0 & \cdots & \cdots & 0 \\
    & \gamma h^{\sG_3} & \cdots & \cdots& 0  \\
    & &\ddots & \cdots & \vdots  \\
    & & & \ddots & \vdots  \\
    0 & & & &  \gamma h^{\sG_{p}}
    \end{bmatrix}.
    \]
    Moreover, by Lemma \ref{lemma:coh-class-iso-holomorphic-class}, the column
    \[
    \begin{bmatrix}
    \theta_{2,p}-\beta_{2,p}\\
    \vdots\\
    \theta_{p-1,p}-\beta_{p-1,p}
    \end{bmatrix}
    \]
    represents a $(0,1)$-form $\bp_{\sG_p^*\otimes \sT} \alpha$, where $\alpha \in A^0(X,\sG_p^*\otimes \bigoplus_{i=2}^{p-1} \sG_i)$ and $\sT$ is the holomorphic  structure on $\bigoplus_{i=2}^{p-1} \sG_i$ induced by $\nabla^\sQ_1$.
    Note that $\alpha$ can also be view as an element in $A^0(X, \sG_p^*\otimes ( \bigoplus_{i=1}^{p-1} \sG_i))$.
    Then there is some $\delta \in A^{0,1}(X,\sG_p^*\otimes \sG_1)$ such that the column
    \[
    \begin{bmatrix}
    \delta - \beta_{1,p}\\
    \theta_{2,p}-\beta_{2,p}\\
    \vdots\\
    \theta_{p-1,p}-\beta_{p-1,p}
    \end{bmatrix}
    \]
    represents $\bp_{\sG_p^*\otimes \sH} \alpha$, where $\sH$ is the holomorphic vector bundle structure on $\bigoplus_{i=1}^{p-1} \sG_i$ induced by $\nabla^\sQ_1$.

    We define the following connection
    \[
    \nabla^{\sQ}_3 =
    \begin{bmatrix}
    \nabla_1 & \theta_{1,2} & \cdots & \cdots & \theta_{1,p-1} & \delta \\
    & \nabla_2  & \cdots & \cdots& \theta_{2,p-1} & \theta_{2,p} \\
    & &\ddots & \cdots & \vdots & \vdots \\
    & & & \ddots & \vdots & \vdots \\
    & & & & \nabla_{p-1} & \theta_{p-1,p} \\
    0 & & & & & \nabla_p
    \end{bmatrix}.
    \]
    Then $(\nabla^{\sQ}_3)^{(0,1)})^2=0$ and it defines a holomorphic structure on $\sQ$ which is isomorphic to the one defined by $\nabla^\sQ_1$ by Lemma \ref{lemma:coh-class-iso-holomorphic-class}.
    Moreover, we have
    \[
    (\nabla^{\sQ}_3)^2  =
    \begin{bmatrix}
    \gamma h^{\sG_1} &  0 & \cdots & \cdots & 0 & \xi \\
    & \gamma h^{\sG_2} & \cdots & \cdots& 0 & 0 \\
    & &\ddots & \cdots & \vdots & \vdots \\
    & & & \ddots & \vdots & \vdots \\
    & & & & \gamma h^{\sG_{p-1}} &  0 \\
    0 & & & & & \gamma h^{\sG_p}
    \end{bmatrix}
    \]
    where
    $$\xi= \nabla^{\sG_p^*\otimes \sG_1}(\delta) + \sum_{i=1}^{p-1}\theta_{1,i}\circ \theta_{i,p} \in A^{2}(X,\sG_p^*\otimes \sG_1).$$

   \noindent The vanishing of the $(1,i)$-entry  of  $(\nabla^{\sQ}_3)^2$ for $1<i<p$ implies
    \begin{equation}\label{eq-vansh-1}
    \nabla_1 \circ \theta_{1,i} = - (\theta_{1,i} \circ \nabla_i + \sum_{j=2}^{i-1} \theta_{1,j}\circ \theta_{j,i}),
    \end{equation}
    and the  vanishing of the $(i,p)$-entry  for $1<i<p$ implies
    \begin{equation}\label{eq-vansh-2}
    \theta_{i,p} \circ \nabla_p = -(\nabla_i \circ \theta_{i,p}  + \sum_{k=i+1}^{p-1} \theta_{i,k} \circ \theta_{k,p}).
    \end{equation}
    Moreover, the condition that $((\nabla^{\sQ}_3)^{(0,1)})^2=0$ shows that
    \begin{equation}\label{eq-vansh-3}
    \bp_{\sG_p^*\otimes \sG_1}(\delta) + \sum_{i=1}^{p-1}\theta_{1,i}\circ \theta_{i,p} =0 \in A^{2}(X,\sG_p^*\otimes \sG_1),
    \end{equation}
    and we deduce
    \begin{equation}\label{eq-vansh-4}
    \xi = (\nabla^{\sG_p^*\otimes \sG_1})^{1,0}(\delta).
    \end{equation}

    Now we compute that
    \begin{eqnarray*}
        \bp_{\sG_p^*\otimes \sG_1}((\nabla^{\sG_p^*\otimes \sG_1})^{1,0}(\delta))
        &=& \bp_{\sG_p^*\otimes \sG_1} (\nabla^{\sG_p^*\otimes \sG_1}(\delta))\\
        &=&-  \nabla^{\sG_p^*\otimes \sG_1}(\bp_{\sG_p^*\otimes \sG_1}(\delta))\\
        &\stackrel{(\ref{eq-vansh-3})}{=}&\nabla^{\sG_p^*\otimes \sG_1}(\sum_{i=1}^{p-1}\theta_{1,i}\circ \theta_{i,p})\\
        &=& \sum_{i=1}^{p-1}\nabla_1\circ \theta_{1,i} \circ \theta_{i,p}- \sum_{i=1}^{p-1} \theta_{1,i}\circ\theta_{i,p}\circ
        \nabla_p\\
        &\stackrel{(\ref{eq-vansh-1}),\ (\ref{eq-vansh-2})}{=}&   \sum_{i=1}^{p-1} \left(- (\theta_{1,i} \circ \nabla_i + \sum_{j=2}^{i-1} \theta_{1,j}\circ \theta_{j,i}) \circ \theta_{i,p}\right) \\
        && -\sum_{i=1}^{p-1} \left(-\theta_{1,i}\circ (\nabla_i \circ \theta_{i,p}  + \sum_{k=i+1}^{p-1} \theta_{i,k} \circ \theta_{k,p})\right)\\
        &=&   -\sum_{i=1}^{p-1}  \sum_{j=2}^{i-1} \theta_{1,j}\circ \theta_{j,i} \circ \theta_{i,p}  \\
        && +\sum_{i=1}^{p-1}\sum_{k=i+1}^{p-1}\theta_{1,i}\circ  \theta_{i,k} \circ \theta_{k,p}\\
        &=&0.
    \end{eqnarray*}
    We note that  $\sG_p^*\otimes \sG_1$ is Hermitian flat by Corollary \ref{DUY}. Hence by $\partial\bp$-Lemma, there exists  $\mu \in A^0(X,\sG_p^*\otimes \sG_1)$ such that
    \begin{equation*}
    (\nabla^{\sG_p^*\otimes \sG_1})^{1,0}(\delta) =(\nabla^{\sG_p^*\otimes \sG_1})^{1,0} ( \bp_{\sG_p^*\otimes
        \sG_1}(\mu))= \nabla^{\sG_p^*\otimes \sG_1} ( \bp_{\sG_p^*\otimes
        \sG_1}(\mu)).
    \end{equation*}
    Thanks to equation (\ref{eq-vansh-4}), this implies that  $$ \xi = \nabla^{\sG_p^*\otimes \sG_1} ( \bp_{\sG_p^*\otimes \sG_1}(\mu)).$$
    Let  $\theta_{1,p}=\delta- \bp_{\sG_p^*\otimes \sG_1}(\mu)$.  Then we check that
    $$
    \nabla^{\sG_p^*\otimes \sG_1}(\theta_{1,p}) + \sum_{i=1}^{p-1}\theta_{1,i}\circ \theta_{i,p}
    = \xi  - \nabla^{\sG_p^*\otimes \sG_1}(\bp_{\sG_p^*\otimes \sG_1}(\mu)) = 0.
    $$
    This completes the proof of Lemma \ref{lemma:connection-construction}.
\end{proof}

Now we are ready to prove Theorem \ref{thm:Kahler:proj-flat-connection} and  Theorem \ref{thm:num-proj-flat=proj-flat}.

\begin{proof}[Proof of Theorem \ref{thm:Kahler:proj-flat-connection}]
    Let $r_i$ be the rank of $\sG_i = \sE_i/\sE_{i+1} $.
    As in Lemma \ref{lemma:unique-proj-flat-connection}, we denote the projectively Hermitian  flat metric on $\sG_i$ by $h^{\sG_i}$ and the Chern connection  by $\nabla_i$.
    Then there is a $(1,1)$-form $\gamma$  such that $\nabla_i^2=\gamma h^{\sG_i}$ for all $i=1,..,p$.
    %
    We denote by $\sQ$ the underlying smooth vector bundle of $\sE$.
    To prove the theorem, we will construct a connection $\nabla^{\sF}$ on $\sQ$ in the following upper triangular form
    \[
    \nabla^{\sF} =
    \begin{bmatrix}
    \nabla_1 & \theta_{1,2} & \cdots & \cdots & \theta_{1,p-1} & \theta_{1,p} \\
    & \nabla_2  & \cdots & \cdots& \theta_{2,p-1} & \theta_{2,p} \\
    & &\ddots & \cdots & \vdots & \vdots \\
    & & & \ddots & \vdots & \vdots \\
    & & & & \nabla_{p-1} & \theta_{p-1,p} \\
    0 & & & & & \nabla_p
    \end{bmatrix}
    \]
    such that
    \begin{enumerate}
        \item $\theta_{i,j} \in A^{0,1}(X,\sG_j^* \otimes \sG_i)$  are   global smooth $(0,1)$-forms;
        \item $\nabla^\sF$ defines a holomorphic structure $\sF$ on $\sQ$ which is isomorphic to $\sE$;
        \item $\nabla^\sF$ is projectively flat.
    \end{enumerate}

\noindent
    By partition of unity, we can construct a connection $\nabla^{\sQ}_1$ on $\sQ$ of the form
    \[
    \nabla^{\sQ}_1 =
    \begin{bmatrix}
    \nabla_1 & \beta_{1,2} & \cdots & \cdots & \beta_{1,p-1} & \beta_{1,p} \\
    & \nabla_2  & \cdots & \cdots& \beta_{2,p-1} & \beta_{2,p} \\
    & &\ddots & \cdots & \vdots & \vdots \\
    & & & \ddots & \vdots & \vdots \\
    & & & & \nabla_{p-1} & \beta_{p-1,p} \\
    0 & & & & & \nabla_p
    \end{bmatrix}
    \]
    such that  $\beta_{i,j} \in A^{1}(X,\sG_j^* \otimes \sG_i)$ and that $(\nabla^{\sQ}_1)^{0,1}=\bp_{\sE}$.  By replacing $\beta_{i,j}$ by its $(0,1)$ part, we may assume further that $\beta_{i,j} \in A^{0,1}(X,\sG_j^* \otimes \sG_i)$.

    We  prove it by induction on $p$. If $p=2$, then it follows from Lemma \ref{bb}.
    Assume  $p>2$ and the assertion holds for smaller integers.
    Then by induction hypothesis, we can find $\theta_{i,j} \in A^{0,1}(X,\sG_j^* \otimes \sG_i)$ for $1\leqslant i<j\leqslant p-1$  such that
    \[
    \begin{bmatrix}
    \nabla_1 & \theta_{1,2} & \cdots & \cdots & \theta_{1,p-1} \\
    &\nabla_2  & \cdots & \cdots& \theta_{2,p-1}  \\
    & & \ddots & \cdots & \vdots  \\
    & & & \ddots & \vdots  \\
    0 & & & & \nabla_{p-1}
    \end{bmatrix}^2
    =
    \begin{bmatrix}
    \gamma h^{\sG_1} &  0 & \cdots & \cdots & 0 \\
    & \gamma h^{\sG_2} & \cdots & \cdots& 0  \\
    & &\ddots & \cdots & \vdots  \\
    & & & \ddots & \vdots  \\
    0 & & & &  \gamma h^{\sG_{p-1}}
    \end{bmatrix}.
    \]

    Moreover,  the following connection on $\sQ$
    \[
    \nabla^{\sQ}_2 =
    \begin{bmatrix}
    \nabla_1 & \theta_{1,2} & \cdots & \cdots & \theta_{1,p-1} & \beta_{1,p} \\
    & \nabla_2  & \cdots & \cdots& \theta_{2,p-1} & \beta_{2,p} \\
    & &\ddots & \cdots & \vdots & \vdots \\
    & & & \ddots & \vdots & \vdots \\
    & & & & \nabla_{p-1} & \beta_{p-1,p} \\
    0 & & & & & \nabla_p
    \end{bmatrix}
    \]
    defines a holomorphic structure isomorphic to $\sE$.  The existence of $\nabla^\sF$ then follows from Lemma \ref{lemma:connection-construction}. This completes the proof of  Theorem \ref{thm:Kahler:proj-flat-connection}.
\end{proof}

\begin{proof}[Proof of Theorem  \ref{thm:num-proj-flat=proj-flat}]
It follows from Theorem \ref{thm:num-proj-flat-criterion-Nakayama} and  Theorem \ref{thm:Kahler:proj-flat-connection}.
\end{proof}

\begin{convention}
    In the sequel of this paper, a holomorphic vector bundle $\sE$ over a quasi-projective manifold $X$ is called a projectively flat vector bundle if it admits a projectively flat connection $\nabla$ such that $\nabla^{0,1}=\bar{\partial}_{\sE}$.
\end{convention}

\subsection{Properties of projectively flat vector bundles}
We underline that all  vector bundles in this subsection are  holomorphic and an isomorphism between vector bundles is always an isomorphism of holomorphic vector bundles.

\begin{lemma}
\label{lemma:extension-proj-flat-vec-bundle}
Let $X$ be a smooth quasi-projective variety and let $\sE$ be a vector bundle on $X$. Assume that there is an open subset $X^\circ \subseteq X$ with complement of codimension at least $2$ such that $\sE|_{X^\circ}$ is isomorphic to a projectively flat vector bundle $\sF^\circ$. Then this isomorphism extends to an isomorphism of vector bundles $\sE \to \sF$ such that $\sF$ is projectively flat.
\end{lemma}

\begin{proof}
Let  $Z^\circ = \mathbb{P}(\sF^\circ)$.  Then $Z^\circ$ is given by a representation $\rho^\circ$ of the fundamental group $\pi_1(X^\circ)$ in $\mathrm{PGL}_{r}(\mathbb{C})$, where $r$ is the rank of $\sE$.
Since the complement of $X^\circ$ has codimension at least 2 in $X$, the fundamental groups  $\pi_1(X^\circ)$  and  $\pi_1(X)$ are canonically isomorphic.
Hence $\rho$ descends to a representation $\rho\colon \pi_{1}(X) \to \mathrm{PGL}_r(\mathbb{C})$. Such a representation induces a projective bundle $p\colon Z \to X$ which extends $Z^\circ \to X^\circ$. Let $\sL^\circ = \sO_{\mathbb{P}(\sF^\circ)}(1)$.
Then it extends to a line bundle $\sL$ on $Z$ whose restriction on a general fiber of $p$ is isomorphic to $\sO_{\bbP^{r-1}}(1)$.
Let $\sF = p_*\sL$.
Then $\sF$ is a projectively flat vector bundle such that $\sF|_{X^\circ}=\sF^\circ$.
Since  the complement of $X^\circ$ has codimension at least 2 in $X$, the isomorphism $\sE|_{X^\circ} \to \sF^\circ$ extends to an isomorphism $\sE \to \sF$.
\end{proof}

\begin{lemma}
\label{lemma:proj-flat-vec-bundle-property2}
Let $f:X\to Y$ be a projective bundle over a quasi-projective manifold $Y$. Assume that $\sE$ is a vector bundle on $Y$ such that $f^*\sE$ is isomorphic to a projectively flat vector bundle on $X$. Then $\sE$ is isomorphic to a projectively flat vector bundle on $Y$.
\end{lemma}

\begin{proof}
Up to isomorphism, we may assume that $\sF = f^*\sE$ is projectively flat. Let $Q=\bbP(\sE)$ and $W = \bbP(\sF)$.
Then $W$ is given by a representation $\rho$ of $\pi_1(X)$ in $\mathrm{PGL}_{r}(\mathbb{C})$, where $r$ is the rank of $\sE$.
We remark that $f$ induces an isomorphism between the fundamental groups $\pi_1(X)$ and $\pi_1(Y)$.
Thus such a representation induces a representation $\eta\colon \pi_1(Y) \to \mathrm{PGL}_r(\mathbb{C})$.  Let $\pi\colon P\to Y$ be the projective bundle given by $\eta$. Then by pulling back $P$, we obtain a projective bundle $q\colon Z\to X$.
By construction, we have an isomorphism $Z\to W$ of projective bundles on $X$.

We denote by $\varphi\colon Z\to Q$ the composition of $Z\to W \to Q.$
Then a complete curve $C$ in $Z$ is contracted by $\varphi$ if and only if the image of $C$ in $P$ is a point.
Hence by rigidity lemma, we see that  $\varphi$ descends to a surjective morphism $\psi\colon P \to Q$ on $Y$.
Since both $P$ and $Q$ are projective bundles on $Y$, such a surjective morphism must be an isomorphism.  Note that the tensor product of a projectively flat vector bundle and a line bundle is again projectively flat, $\sE$ is isomorphic to a projectively flat vector bundle.
\end{proof}

The following lemma is a consequence of \cite[Theorem 1.5]{GrebKebekusPeternell2016b}.

\begin{lemma}\label{lemma:trivial-bundle-finite-change}
    Let $Y$ be a projective manifold and  $p \colon Y \to Y'$   a surjective morphism onto a normal projective variety with klt singularities. Let $P\rightarrow Y$ be a  flat projective bundle given by a representation $\pi_1(Y)\rightarrow \PGL_{d+1}(\bbC)$.
    Assume that the fibers of $p$ are simply connected over some open subset  $V\subseteq Y'$ such that $\codim Y'\backslash V \geqslant 2$.
    Then for every  $p$-exceptional curve $C$, there is a smooth   curve $D$ with finite surjective morphism $\mu \colon D \to C$ such that $P$ induces a trivial $\mathbb{P}^d$-bundle structure on $D \times_C P$ over $D$. In particular, if $P\cong \bbP(\sE)$ for some vector bundle $\sE$ on $Y$, then $\mu^*\sE$ is isomorphic to the direct sum of copies of a line bundle on $D$.
\end{lemma}

\begin{proof}
    By \cite[Theorem 1.5]{GrebKebekusPeternell2016b}, there is a quasi-\'etale cover $Z'\to Y'$ such that the algebraic fundamental groups $\hat{\pi}_1(Z'_{\mathrm{reg}})$ and $\hat{\pi}_1(Z')$ are canonically isomorphic. Let $Z$ be a desingularization of the main component of the fiber product $Y\times_{Y'} Z'$.

    \[\begin{tikzcd}[column sep=large, row sep=large]
        Z \rar{h}\dar[swap]{q}   &   Y\dar{p}\\
        Z' \rar   &   Y'
      \end{tikzcd}
    \]

    We note that $P$ is given by a representation of the fundamental group $\pi_1(Y)$ in $\mathrm{PGL}_{d+1}(\mathbb{C})$.
    Since the  fibers of $p$ are simply connected over   $V\subseteq Y'$, such a representation induces a representation of $\pi_1(Y_{\mathrm{reg}})$ in $\mathrm{PGL}_{d+1}(\mathbb{C})$.
    In another word, $P$ induces a flat $\mathbb{P}^d$-bundle $P'$ on the smooth locus of $Y'$.
    By pulling back we obtain a flat $\mathbb{P}^d$-bundle $Q'_1$ on  $Z'_{\mathrm{reg}}$,  given by a representation of $\pi_1(Z'_{\mathrm{reg}})$ in $\mathrm{PGL}_{d+1}(\mathbb{C})$.
    As in \cite[Proof of Theorem 1.14 on page 1990]{GrebKebekusPeternell2016b},   this representation of $\pi_1(Z'_{\mathrm{reg}})$ can descend to a representation  of $\pi_1(Z') $ in $\mathrm{PGL}_{d+1}(\mathbb{C})$.
    Hence there is a flat $\mathbb{P}^d$-bundle $Q'$ on $Z'$ which extends $Q'_1$. We denote by $Q$ the $\mathbb{P}^d$-bundle on $Z$ obtained by pulling back $Q'$. We note that $Q$ is a trivial  $\mathbb{P}^d$-bundle on every fiber of $q$.

    We also remark that the pullback of $P$ by $h$ is  isomorphic to $Q$.  Indeed, by construction, they are isomorphic over some open dense subset $U$ of $Z$. Moreover, the natural morphism $\pi_1(U) \to \pi_1(Z)$ is surjective and  both of these two $\mathbb{P}^d$-bundles are defined by representation of $\pi_1(Z)$ in $\mathrm{PGL}_{d+1}(\mathbb{C})$.

     Let $C$ be a    $p$-exceptional curve. Then there is a   curve $D_1\subseteq Z$ contained in some fiber of $q$ such that  $h(D_1)=C$. Let $D$ be the normalization of $D_1$ and $\mu\colon D\to C$ the induced morphism. Then the pullback of $D\times_C P \cong D\times_{D_1} Q$ is a trivial $\mathbb{P}^d$-bundle  on $D$.
\end{proof}

\vskip 2\baselineskip

\section{Projectivized bundles of vector bundles}


\label{Examples}

Let $X$ be a complex manifold. A projective bundle $P$ on $X$ is a smooth fibration $\varphi\colon P\rightarrow X$ with fibers isomorphic to  $\bbP^d$ for some $d>0$.
A projective bundle is called \emph{insignificant} if it is isomorphic to the projectivized bundle  $\bbP(\sE)$ of some vector bundle $\sE$. This is equivalent to the existence of a line bundle $\sL$ on $P$ whose restriction on every fiber of $\varphi$ is isomorphic to $\sO_{\bbP^d}(1)$.
If we denote by ${\rm Proj}(X)$ the set of isomorphism classes of projective bundles on $X$, then   there is a map $\delta\colon {\rm Proj}(X)\rightarrow H^2(X,\sO_X^{\times})$  such that
\begin{center}
    $P\in {\rm Proj}(X)$ is insignificant $\Leftrightarrow$ $\delta(P)=1\in H^2(X,\sO_X^{\times}).$
\end{center}
In another word, $\delta(P)$ is the obstruction for $P$ to be insignificant.
The image of $\delta$  is called the \emph{Brauer group} of $X$ (see for instance \cite{Elencwajg1985} for more details).
We recall the following theorem.

\begin{thm}\cite[Theorem~1]{Elencwajg1985}\label{thm:Brauer-group}
    Let $P\rightarrow X$ be a projective bundle on a complex manifold $X$. Then the fiber product $P'=P\times_XP$ is an insignificant projective bundle on $P$.
\end{thm}


In this section, we   study insignificant projective bundles $X=\mathbb{P}(\sE)$ whose tangent bundle $T_X$ contains a strictly nef subsheaf $\sF$ and prove a structure theorem of such couples $(\mathbb{P}(\sE), \sF)$ in the first subsection and then provide some examples  in the second subsection.

\subsection{Strictly nef subsheaves of $T_{\mathbb{P}(\sE)}$}

We first  prove the following theorem, which is an analogue of \cite[Lemma 1.2]{CampanaPeternell1998}. It classifies all almost nef locally free subsheaves of the tangent bundle of an insignificant projective bundle provided stronger positivity of the restrictions to fibers.


\begin{thm}
    \label{Proj-bundle}
    Let $Y$ be a projective manifold of dimension $m>0$ and let $\sE$ be a vector bundle of rank $d+1$ on $Y$. Denote by $X$ the projective bundle $\bbP(\sE)$ and by $p\colon X\rightarrow Y$ the natural projection.
    Assume that there is an almost nef locally free   subsheaf $\sF \subseteq T_{X}$ of rank $r$ such that its negative locus $\bbS(\sF)$ does not dominate $Y$.
    \begin{enumerate}
        \item If $\sF\vert_F=T_F$ for a general fiber $F$ of $p$, then $\sE$ is numerically projectively flat  and $\sF=T_{X/Y}$.

        \item If $\sF\vert_F\cong  \sO_{\bbP^n}(1)^{\oplus r}$ for each fiber $F$ of $p$, then there exists a numerically projectively flat subbundle $\sM$ of $\sE^*$ such that $\sF\cong  p^*\sM\otimes\sO_{\bbP(\sE)}(1)$.
    \end{enumerate}
\end{thm}

\begin{proof}
	By our assumption, in both cases, the restriction of $\sF$ to a general fiber $F$ of $p$ is ample. Therefore $\sF\vert_F$ is contained in $T_F$ for general $F$. Consequently, $\sF$ is contained in the relative tangent bundle $T_{X/Y}$.
	
    \textit{Proof of (1).} Let $S$ be the support of the torsion sheaf $T_{X/Y}/\sF$. Then $S$ is a closed  subvariety of $X$ which does not dominate $Y$. Consider the relative Euler sequence
    \begin{equation}
    \label{eq:relative-euler-sequence-proof-PF}
    0\rightarrow \sO_X\rightarrow p^*\sE^*\otimes\sO_{\bbP(\sE)}(1)\rightarrow T_{X/Y}\rightarrow 0.
    \end{equation}
    Since the inclusion $\sF\hookrightarrow T_{X/Y}$ is generically surjective and since $\sF$ is almost nef, by Proposition \ref{Almost-nef-properties} (1), $T_{X/Y}$ is almost nef and $\bbS(T_{X/Y})\subseteq S\cup\bbS(\sF)$.
    Then it follows from Proposition \ref{Almost-nef-properties} (4) that the vector bundle $p^*\sE^*\otimes\sO_{\bbP(\sE)}(1)$ is almost nef with negative locus contained in $S\cup \bbS(\sF)$. In particular, from Proposition \ref{Almost-nef-properties} (2), its determinant
    \[\det(p^*\sE^*\otimes\sO_{\bbP(\sE)}(1))=\sO_{\bbP(\sE)}(d+1)\otimes p^*\det(\sE^*)\]
    is almost nef with negative locus contained in $S\cup \bbS(\sF)$.
    This implies that the normalized tautological class $\Lambda_{\sE}$ is almost nef with $\bbS(\Lambda_{\sE})\subseteq S\cup\bbS(\sF)$.
    Since   $S\cup \bbS(\sF)$ does not dominate $Y$, the $\bbQ$-twisted vector bundle $\sE\hspace{-0.8ex}<\hspace{-0.8ex}-\mu(\sE) \hspace{-0.8ex}>$ is almost nef by Proposition \ref{Almost-nef-properties} (3), where $\mu(\sE)$ is the average first Chern class of $\sE$.
    Hence $\sE$ is numerically projectively flat by Theorem \ref{Num-Projectily-flatness-criterion}.


    Next we show that $\sF=T_{X/Y}$.   The induced morphism $\det(\sF)\rightarrow \det(T_{X/Y})$ implies that there exists an effective divisor $D$ such that $\det(T_{X/Y})\cong  \det(\sF)\otimes\sO_X(D)$.
    We note that $$c_1(T_{X/Y})-c_1(\sF) \equiv_p 0, $$ since the relative Picard number of $X$ over $Y$ is one.
    Hence there exists a line bundle $\sL$ on $Y$ such that $p^*\sL=\sO_X(D)$.
    By the relative Euler sequence \eqref{eq:relative-euler-sequence-proof-PF},  we have
    \[(d+1)\Lambda_{\sE}-c_1(p^*\sL)=c_1(T_{X/Y})-c_1(\sO_X(D))=c_1(\sF),\]
    which is almost nef since $\sF$ is. Moreover, since its negative locus is contained in $\bbS(\sF)$, the $\bbQ$-twisted vector bundle $(\sE\otimes \sL^*)\hspace{-0.8ex}<\hspace{-0.8ex} -\mu(\sE)\hspace{-0.8ex}>$ is almost nef. Then taking the determinant shows that $\sL^*$ is pseudoeffective.
    This implies that   $D=0$ and $\det(\sF)=\det(T_{X/Y})$. Since $\sF$ and $T_{X/Y}$ are vector bundles of the same rank, by \cite[Lemma 1.20]{DemaillyPeternellSchneider1994}, we have $\sF=T_{X/Y}$.\\

    \textit{Proof of (2).} The existence of $\sM \hookrightarrow \sE^*$ follows from Lemma \ref{lemma:factorization} below.
     Since $\sF$ is almost nef, by Proposition \ref{Almost-nef-properties} (2), the determinant
    \[\det(\sF)=p^*\det(\sM)\otimes\sO_{\bbP(\sE)}(r)\]
    is almost nef on $\bbP(\sE)$ with negative locus contained in $\bbS(\sF)$. So the $\bbQ$-twisted vector bundle $\sE\hspace{-0.8ex}<\hspace{-0.8ex} \mu(\sM) \hspace{-0.8ex}>$ is almost nef (see Proposition \ref{Almost-nef-properties} (3)). Since the natural morphism $\sE\rightarrow \sM^*$ is generically surjective, by Proposition \ref{Almost-nef-properties} (1), the $\bbQ$-twisted vector bundle $\sM^*\hspace{-0.8ex}<\hspace{-0.8ex} \mu(\sM) \hspace{-0.8ex}>$ is almost nef as well. Hence, it follows from Theorem \ref{Num-Projectily-flatness-criterion} that $\sM^*$ is numerically projectively flat, and so is $\sM$ by Lemma \ref{lemma:prop-num-proj-flat}.

    Next we  show that $\sM$ is saturated in $\sE^*$. Fix an ample divisor $A$ on $X$ and let $\sG$ be the maximal  destabilizing subsheaf of $\sE^*$ with respect to $A$. Since $\sM$ is numerically projectively flat, both $\sM^*$ and $\sM$ are $A$-semistable by Theorem \ref{thm:num-proj-flat-criterion-Nakayama}. In particular, we have
    \[\mu_{A}^{\min}(\sM)=\mu_A(\sM).\]
    On the other hand, since $\sM$ is a subsheaf of $\sE^*$, we have
    \[\mu_A(\sM)\leqslant \mu_A^{\max}(\sE^*)=\mu_A(\sG).\]
    Recall that the $\bbQ$-twisted vector bundle $\sE\hspace{-0.8ex}<\hspace{-0.8ex} \mu(\sM) \hspace{-0.8ex}>$ is almost nef.
    Since the  natural morphism $\sE\rightarrow \sG^*$ is generically surjective, by Proposition \ref{Almost-nef-properties} (1),   the $\bbQ$-twisted  sheaf $\sG^*\hspace{-0.8ex}<\hspace{-0.8ex} \mu(\sM) \hspace{-0.8ex}>$ is almost nef. Thus the $\bbQ$-Cartier divisor class
    \[
    c_1(\sG^*)+\rk(\sG^*)\mu(\sM)
    \]
    is pseudoeffective on $Y$.  It yields
    $$
    (c_1(\sG^*)+\rk(\sG^*)\mu(\sM))\cdot A^{n-1} \geqslant 0.
    $$  
    This implies
    \[
    \mu_A(\sM)=\mu(\sM)\cdot A^{n-1}\geqslant -\frac{1}{\rk(\sG^*)}c_1(\sG^*)\cdot A^{n-1}= -\mu_A(\sG^*)=\mu_A(\sG).
    \]
    Hence, we have $\mu_A(\sM)=\mu_A(\sG)$. By the definition of $\sG$, we see that $\sM$ is contained in $\sG$. Let $\overline{\sM}$ be the saturation of $\sM$ in $\sE^*$. Then  $\overline{\sM}$ is   contained in $\sG$. Moreover, as $\mu_A(\sM)\leqslant \mu_A(\overline{\sM})\leqslant \mu_A(\sG)$, we deduce that
    \[\mu_A(\sM)=\mu_A(\overline{\sM}).\]
    This implies that  the inclusion $\det(\sM)\rightarrow \det(\overline{\sM})$ is an isomorphism. Then   applying \cite[Lemma 1.20]{DemaillyPeternellSchneider1994} shows that $\sM$ is saturated in $\sE^*$.

    Finally, we choose an arbitrary point $x\in Y$. Let $C\subseteq Y$ be a general complete smooth curve passing through $x$. Then $\sM\vert_C$ is again a subsheaf of $\sE^*\vert_C$.
     Now replacing $Y$ by $C$ in the argument above, we can conclude that $\sM\vert_C$ is saturated in $\sE^*\vert_C$.
     Since $C$ is a curve, $\sM\vert_C$ is actually a subbundle of $\sE^*\vert_C$. This shows that  $\sM$ is a subbundle of $\sE^*$.
\end{proof}

\begin{lemma}\label{lemma:factorization}
Let $Y$ be a smooth quasi-projective variety and $\sE$ a vector bundle on $Y$. Let  $X=\mathbb{P}(\sE)$ and  let $p\colon X\to Y$ be the natural fibration. Assume that there is a locally free subsheaf $\sF \hookrightarrow T_{X/Y}$ such that $\sF|_F\cong \sO_{F}(1)^{\oplus r}$ for any fiber $F$ of $p$. Then there is a  vector bundle $\sM$ on $Y$ such that $\sF = p^*\sM\otimes \sO_{\mathbb{P}(\sE)}(1)$. Moreover, there is an injective morphism $\sM \hookrightarrow \sE^*$ which induces the inclusion $\sF \hookrightarrow T_{X/Y}$ via the relative Euler sequence.
\end{lemma}

\begin{proof}
Set $\sM=p_*(\sF \otimes \sO_{\mathbb{P}(\sE)}(-1))$. Then $\sF = p^*\sM\otimes \sO_{\bbP(\sE)}(1).$
Tensoring the relative Euler sequence of $\bbP(\sE)$ with $\sF^*$, we  derive the following short exact sequence of vector bundles
    \[0\rightarrow \sF^*\rightarrow p^*\sE^*\otimes\sO_{\bbP(\sE)}(1)\otimes\sF^*\rightarrow T_{X/Y}\otimes\sF^*\rightarrow 0.\]
    For any $i\geqslant 0$, we have
    \[R^i p_*\sF^*=\sM^*\otimes R^i p_*\sO_{\bbP(\sE)}(-1)=0.\]
    Thus Leray Spectral sequence   implies that  $H^1(X,\sF^*)=0.$
    As a consequence,  the inclusion $\sF\hookrightarrow T_{X/Y}$ can be lifted to an inclusion $\sF\hookrightarrow p^*\sE^*\otimes\sO_{\bbP(\sE)}(1)$.
    We then obtain an  injective morphism $p^*\sM \hookrightarrow p^*\sE^*$. By taking the direct image, we obtain an inclusion  $\sM\hookrightarrow \sE^*$.
\end{proof}

\begin{lemma}\phantomsection\label{Pn} 	
    \begin{enumerate}
    	\item Let $\sF$ be a strictly nef vector bundle of rank $r$ on $\bbP^n$. Assume that there exists an injective morphism $\sF\hookrightarrow T_{\bbP^n}$. Then $\sF$ is isomorphic to $\sO_{\bbP^n}(1)^{\oplus r}$ or $T_{\bbP^n}$.
    	
    	\item Let $p\colon X \rightarrow Y$  be a projective bundle. Let $\sF$ be a vector bundle on $X$  which is strictly nef on every fiber of $p$. Assume further that $\sF|_F \cong \sO_{F}(1)^{\oplus r}$ for a fiber $F$ of $p$. Then for every fiber $F'$ of $p$, we have $\sF|_{F'} \cong \sO_{F'}(1)^{\oplus r}$.
    \end{enumerate}
  
\end{lemma}

\begin{proof}
    For (1), thanks to \cite[Theorem 4.2]{AproduKebekusPeternell2008}, it suffices to show that the restriction $\sF\vert_{\ell}$ on an arbitrary  line $\ell \subseteq \bbP^n$ is ample. Since  $\ell$ is a rational curve, the strictly nef bundle $\sF\vert_{\ell}$ must be ample.
    
    For (2), let $F'$ be an arbitrary fiber of $p$ and $C$ be a  line in $F'$.
Then $\sF|_C$ is a strictly nef vector bundle of rank $r$. Moreover, by assumption the degree of $\sF|_C$ is equal to $r$.
Hence we obtain $\sF|_C \cong \sO_{C}(1)^{\oplus r}$. As $C$ is arbitrary, this implies that $\sF|_{F'} \cong \sO_{F'}(1)^{\oplus r}$ by \cite[Proposition 1.2]{AndreattaWisniewski2001}.
\end{proof}

\begin{remark}
	\label{remark:strictly-nef}
	By the lemma above, if we assume that $\sF$ is strictly nef in Theorem \ref{Proj-bundle}, then  either $\sF\vert_F=T_F$ for a general fiber $F$ of $p$, or $\sF\vert_F\cong  \sO_{\bbP^n}(1)^{\oplus r}$ for any fiber $F$ of $p$. 
\end{remark}

As an immediate corollary of Theorem \ref{Proj-bundle}, we obtain  the following result.

\begin{cor}\label{cor:proj-vector-bundle-p1-base}
Let $\sE$ be a vector bundle on $\mathbb{P}^1$. Then the tangent bundle $T_{\mathbb{P}(\sE)}$ does not contain any strictly nef locally free subsheaves.
\end{cor}

\begin{proof}
We write $Y=\mathbb{P}^1$ and $X=\mathbb{P}(\sE)$.
Assume by contradiction that there is a strictly nef locally free subsheaf $\sF$ of $T_{X}$. Then $\sF$ is contained in $T_{X/Y}$.
As pointed out in Remark \ref{remark:strictly-nef}, by Lemma \ref{Pn}, we are in one of the  situations of  Theorem \ref{Proj-bundle}.

We note that  every numerically projectively flat vector bundle on $Y$ is the direct sum of copies of a line bundle. Hence, if we are in the first case of the theorem, then $X\cong \mathbb{P}^n\times Y$. One readily check that $T_{X/Y}$ is not strictly nef.
Now we assume the second case of Theorem \ref{Proj-bundle}. We write $\sM=\sL^{\oplus r}$. Then we have  a surjective morphism $$\sE \to \sM^* \to \sL^*,$$ where $\sM^* \to \sL^*$ one of the canonical projections.
The composition of the morphisms above induce a section $\sigma\colon Y\to {\mathbb{P}(\sE)}$ such that $$\sigma^*\sO_{\mathbb{P}(\sE)}(1) \cong \sL^*.$$ As a consequence  $\sigma^*\sF\cong \sO_Y^{\oplus r}$.  Hence $\sF$ is not strictly nef.
\end{proof}



We conclude this subsection with the following   proposition.

\begin{prop}
	\label{prop:smothness-image-subbundle}
	Let $T$ be a complex manifold and let $X\rightarrow T$ be a $\bbP^d$-bundle. Let $Z$ be the fiber product $X\times_T X\cong \bbP(\sE)$, where $\sE$ is a vector bundle over $X$ (see Theorem \ref{thm:Brauer-group}).
	\[
	\begin{tikzcd}[column sep=large, row sep=large]
	X\dar[swap]{\varphi} & Z  \cong \bbP(\sE) \lar{q}\dar[swap]{\pi} \\
	T &
	X\lar{p=\varphi}
	\end{tikzcd}
	\]
	Assume that there exists a vector bundle $\sF$ over $X$ and a subbunlde $\sM$ of $\sE^*$ such that 
	\[
	q^*\sF\cong \pi^*\sM\otimes \sO_{\bbP(\sE)}(1).
	\]
	Let $M\subseteq Z$ be the projective subbundle $\bbP(\sM^*) \subseteq \bbP(\sE)$. Then $q(M)$ is a projective bundle over $T$.
\end{prop}

\begin{proof}
	The problem is local on $T$, hence we can suppose that $T\subseteq \bbC^m$ is a small polydisc  so that   $X\cong \bbP(\sV)$ for some vector bundle $\sV$ on  $T$. 
	Then there exists a line bundle $\sL$ on  $X$ such that $p^*\sV\cong \sE\otimes \sL$. 
	By  replacing $\sE$ with $\sE\otimes \sL$ and   $\sM$ with $\sM\otimes \sL^*$, we may assume that $\sL$ is trivial. That is $p^*\sV\cong \sE$.
	
	Since the restrictions of $q^*\sF$ to the fibers of $\pi$ are isomorphic to direct sum of copies of $\sO_{\bbP^d}(1)$, the same holds for the restrictions of $\sF$ to the fibers of $\varphi$. 
	We define the following  the vector bundle on $T$
	\[
	\sQ=\varphi_*(\sF\otimes \sO_{\bbP(\sV)}(-1)).
	\]
	Then by assumption, we have
	\[
	q^*(\varphi^*\sQ\otimes \sO_{\bbP(\sV)}(1))\cong q^*\sF\cong \pi^*\sM\otimes \sO_{\bbP(\sE)}(1).
	\] 
	As $q^*\sO_{\bbP(\sV)}(1)\cong \sO_{\bbP(\sE)}(1)$, we deduce that $\pi^*p^*\sQ\cong \pi^*\sM$ and hence
	\[
	 p^*\sQ\cong \sM.
	\]
	 Since $p^*\sV^*\cong \sE^*$, the subbundle structure  $\sM \hookrightarrow \sE^*$ induces a subbundle structure $\sQ \hookrightarrow \sV^*$. Denote by $Q$ the projective  subbundle $\bbP(\sQ^*) \subseteq X \cong \bbP(\sV )$. 
	 Then $M$ is isomorphic to $Q\times_T X$ and consequently $q(M)=Q$. This finishes the proof.
\end{proof}

\subsection{Examples}

\label{Subsection:examples}


In this subsection, our goal is to extend Mumford's example to higher dimension.
A key ingredient is the following theorem due to Subramanian.

\begin{thm}\cite[Lemma 3.2 and Theorem 6.1]{Subramanian1989}\label{Subramanian}
    Let $C$ be a smooth curve of genus $g\geqslant 2$. Then for any $r\geqslant 2$, there exists a Hermitian flat vector bundle $\sE$ of rank $r$ such that the tautological line bundle $\sO_{\bbP(\sE)}(1)$ is ample when restricted to a proper subvariety of $\bbP(\sE)$. In particular, $\sE$ is strictly nef.
\end{thm}

By using this theorem, we construct  two examples.  Fix  a smooth  curve $C$ of genus $g \geqslant 2$. Let  $r\geqslant 2$ and  $\sE$  a vector bundle of rank $r$ provided in Theorem \ref{Subramanian}.

\begin{example}[Strictly nef subbundles]
    \label{example:subbundles}
    Let $X=\bbP(\sE)$. Then we have the following relative Euler exact sequence
    \[0\rightarrow \sO_X\rightarrow p^*\sE^*\otimes\sO_{\bbP(\sE)}(1)\rightarrow T_{X/C}\rightarrow 0,\]
    where $p\colon X=\bbP(\sE)\rightarrow C$ is the natural projection.
    We claim that $$\sE'\coloneqq p^*\sE^*\otimes \sO_{\bbP(\sE)}(1)$$ is strictly nef.
    Indeed, let $\nu\colon C'\rightarrow X$ be a finite morphism from a  smooth complete curve $C'$ to $X$ and let $\nu^*\sE'\rightarrow \sL$ be a quotient line bundle. If $\nu(C')$ is contained in the fibers of $p$, then $\nu^*\sE'$ is ample and $\sL$ is still ample.
    Assume that $\nu(C')$ is not contained in the fibers of $p$. Then the composition $\nu'\colon C'\rightarrow X\rightarrow C$ is a finite morphism. Moreover, we have
    \[\nu^*\sE'\cong  \nu'^*\sE^*\otimes \nu^*\sO_{\bbP(\sE)}(1).\]
    This implies that $\sL\otimes \nu^*\sO_{\bbP(\sE)}(-1)$ is a quotient line bundle of $\nu'^*\sE^*$. Note that $\sE^*$ is numerically flat and therefore nef. Thus we have
    \[
    \deg(\sL)+\deg(\nu^*\sO_{\bbP(\sE)}(-1))\geqslant 0.
    \]
    However,  since $\sO_{\bbP(\sE)}(1)$ is strictly nef, we deduce that $\deg(\sL)>0$.  Therefore $\sE'$ is strictly nef by \cite[Proposition 2.1]{LiOuYang2019}. Hence, $\sF = T_{X/C}$ is a strictly nef subbundle of $T_X$ (see \cite[Propostion 2.2]{LiOuYang2019}).
\end{example}

\begin{example}[Strictly nef subsheaves which are not subbundles]\label{example:not-subbundles}
    We consider the following extension of vector bundles
    \[0\rightarrow \sQ\rightarrow \sG\rightarrow \sE^*\rightarrow 0,\]
    where $\sQ$ is a nef vector bundle of positive rank.    Since $\sE^*$ is Hermitian flat, it is numerically flat. In particular, $\sE^*$ is nef and so is $\sG$ (see \cite[Proposition 1.15]{DemaillyPeternellSchneider1994}).
    Let $X=\bbP(\sG)$ and  $p\colon X= \bbP(\sG)\rightarrow C$ the natural projection. Then we have the following relative Euler sequence
    \[0\rightarrow \sO_X\rightarrow p^*\sG^*\otimes\sO_{\bbP(\sG)}(1)\rightarrow T_{X/C}\rightarrow 0.\]
    Since $\sE$ is a subbundle of $\sG^*$, it follows that
    $$ \sF \coloneqq p^*\sE\otimes \sO_{\bbP(\sG)}(1)$$ is a subbundle of $p^*\sG^*\otimes \sO_{\bbP(\sG)}(1)$.
    As in Example \ref{example:subbundles} above, one can  show that  $\sF$ is strictly nef.
    Moreover,  note that the composition
    \[\sF \rightarrow p^*\sG^*\otimes \sO_{\bbP(\sG)}(1)\rightarrow T_{X/C}\]
    is injective, it follows that $\sF$ is a strictly nef locally free subsheaf of $T_X$. On the other hand, since the restriction of $\sF$ to fibers of $p$ is isomorphic to $\sO_{\bbP^d}(1)^{\oplus r}$,  $\sF$ is not a subbundle of $T_{X/C}$.
\end{example}

\vskip 2\baselineskip

\section{Structures of MRC fibrations}


\label{section:uniruled}

 In this section we prove the following proposition.

\begin{prop}\label{MRC-fibration}
    Let $X$ be a projective manifold. Assume that $T_X$ contains a  locally free strictly nef subsheaf $\sF$. Then $X$ is uniruled. 
    
    Moreover, denote by  $\varphi\colon X\dashrightarrow T$  its MRC fibration. Then there exists an open subset $X^\circ\subseteq X$ with $\codim(X\setminus X^\circ)\geqslant 2$ such that the restriction
    \[\varphi^\circ\coloneqq\varphi\vert_{X^\circ}\colon X^\circ\rightarrow T^\circ\subseteq T\]
    is a $\bbP^d$-bundle. In particular, the restriction $\sF\vert_{X^\circ}$ is contained in $T_{X^\circ/T^\circ}$.
\end{prop}

We  first show that such a projective manifold $X$ must be uniruled (see Corollary \ref{Uniruledness}). We remark that \emph{a priori} this is not straightforward since a strictly nef vector bundle may also be numerically flat, and Miyaoka's criterion cannot be applied in this case.
The uniruledness in Proposition \ref{MRC-fibration}  is  a consequence of the following theorem, which itself may be of independent interest.


\begin{thm}\label{Almost-nef-non-uniruled}
    Let $X$ be a non-uniruled projective manifold. Assume that there is a non-zero map $\sigma\colon \sF\rightarrow T_X$, where $\sF$ is an almost nef coherent sheaf over  $X$. Let $\widehat{\sQ}$ be the reflexive hull of the image $\sQ$ of $\sigma$. Then $\widehat{\sQ}$ is an involutive subbundle of $T_X$, and  is numerically flat  with torsion determinant bundle. Furthermore, there exist a finite \'etale cover  $\gamma\colon \widetilde{X}\rightarrow X$ and an almost holomorphic map $g\colon \widetilde{X}\dashrightarrow Y$ whose general fibers are   abelian varieties such that the restriction of $\gamma^*\widehat{\sQ}$ to a general fiber of $g$ is a linear foliation.
\end{thm}

\begin{proof}
    Let $\overline{\sQ}$ be the saturation of $\sQ$ in $T_X$. By the definition of reflexive hull, we have  natural inclusions $\sQ\hookrightarrow\widehat{\sQ}\hookrightarrow\overline{\sQ}$, which induce  an injection
    \begin{equation}\label{Eq-determinants}
     \det(\widehat{\sQ})\hookrightarrow\det(\overline{\sQ}).
    \end{equation}
    Since $\sF$ is almost nef and the composition $\sF\rightarrow \widehat{\sQ}\rightarrow \overline{\sQ}$ is generically surjective, Proposition \ref{Almost-nef-properties} shows that   both $\widehat{\sQ}$ and $\overline{\sQ}$ are almost nef, so are $\det(\widehat{\sQ})$ and $\det(\overline{\sQ})$.   On the other hand, as $X$ is not uniruled, by \cite[Theorem 2.6]{BoucksomDemaillyPuaunPeternell2013}, both $\det(\widehat{\sQ})^*$ and $\det(\overline{\sQ})^*$ are pseudoeffective. Therefore, by  \cite[Lemma 6.5]{Peternell1994}, we have 
    \[
    c_1(\widehat{\sQ})= c_1(\overline{\sQ})=0.
    \]
    As a consequence, the natural morphism $\det(\widehat{\sQ})\hookrightarrow\det(\overline{\sQ})$ is  an isomorphism.   We can now apply \cite[Theorem 5.2]{LorayPereiraTouzet2018} to conclude that $\overline{\sQ}$ is a regular foliation  with torsion canonical bundle. Then, from \cite[Lemma 1.20]{DemaillyPeternellSchneider1994}, we obtain  $\widehat{\sQ} = \overline{\sQ}$ .

    Finally, since $\widehat{\sQ}$ is almost nef and $c_1(\widehat{\sQ})=0$, \cite[Theorem 1.8]{HoeringPeternell2019} says that $\widehat{\sQ}$ is a numerically flat vector bundle. In particular,  $c_2(\widehat{\sQ})=0$.  The remaining part of the theorem then follows from   \cite[Theorem C]{PereiraTouzet2013}.
\end{proof}

As an application, we obtain the following criterion for uniruledness.

\begin{cor}\label{Uniruledness}
    Let $\sF$ be a strictly nef coherent sheaf on a projective manifold $X$. Assume that there exists a non-zero map $\sigma\colon\sF\rightarrow T_X$ with image $\sQ$. Then $X$ is uniruled. Moreover, there exists an open subset $X^\circ\subseteq X$ and a $\mathbb{P}^d$-bundle $\varphi^\circ\colon X^\circ\rightarrow T^\circ$ over a smooth base $T^\circ$ such that $\sQ\vert_{X^\circ}\subseteq T_{X^\circ/T^\circ}$.
\end{cor}

\begin{proof}
    We first assume by contradiction that $X$ is not uniruled.  Let $\widehat{\sQ}$ be the reflexive hull of $\sQ$.  After Theorem \ref{Almost-nef-non-uniruled}, replacing $X$ by some finite \'etale cover if necessary, we may assume that there is an almost holomorphic map $g\colon X \dashrightarrow Y$ whose general fibers are   abelian varieties. Moreover,  the restriction of $ \widehat{\sQ}$ on a general fiber $F$ of $g$ is a linear foliation.

    Denote by $S$ the support of $\widehat{\sQ}/\sQ$. Then $S$ is contained in the singular locus $\mathrm{Sing}(\sQ)$ of $\sQ$. In particular, $S$ has codimension at least $2$ as $\sQ$ is torsion free.  As a consequence, the intersection $F\cap S$ also has codimension at least $2$ in $F$.  Therefore, if $C$ is a complete intersection curve of general very ample divisors in $F$,  we have $\widehat{\sQ}\vert_C = \sQ|_C$. The latter is strictly nef by Proposition \ref{Image-nef}.    This contradicts to the fact that $\widehat{\sQ}\vert_A$ is a linear foliation.

    Finally, since $X$ is uniruled, it carries a covering family $\cV$ of minimal rational curves.
    As $\sQ$ is locally free in codimension one, by \cite[II, Proposition 3.7]{Kollar1996},   $\sQ$ is locally free along a general member $D$ of $\cV$. Note that $\sF|_D$ is strictly nef  by Proposition \ref{BK criterion}, so is $\sQ\vert_D$   by Proposition \ref{Image-nef}.
    This in turn implies that $\sQ|_D$ is an ample vector bundle  as $D$ is a rational curve. The remaining part of the corollary then  follows   from \cite[Proposition 2.7]{AraujoDruelKovacs2008}.
\end{proof}

\begin{remark}\label{rmk:rc-quotient}
    In view of the proof of \cite{AraujoDruelKovacs2008} (see also \cite[Theorem 3.4]{Araujo2006} or \cite[Theorem 2.1]{Liu2019}), the almost holomorphic map on $X$ induced by $\varphi^\circ$ is nothing but   the $\cV$-rationally connected quotient of $X$. Furthermore, every rational curve in $\cV$  meeting $X^\circ$ is a line in a fiber of $\varphi^\circ$.
\end{remark}

Now we can conclude Proposition \ref{MRC-fibration}.

\begin{proof}[Proof of Proposition \ref{MRC-fibration}]
    By Corollary \ref{Uniruledness}, $X$ is uniruled and it carries a covering family $\cV$ of minimal rational curves. As pointed out in Remark \ref{rmk:rc-quotient}, the $\cV$-rationally connected quotient $\varphi^\circ\colon X^\circ\rightarrow T^\circ$ is a $\mathbb{P}^d$-bundle and the rational curves parameterized by $\cV$ meeting $X^\circ$ are lines in  fibers of $\varphi^\circ$. Moreover, since $\cV$ is unsplit by Lemma \ref{lemma:unsplit} below, thanks to \cite[Theorem 3.4]{Araujo2006} (see also \cite[Theorem 2.1]{Liu2019}), $\varphi^\circ$ can be extended in codimension one; that is, we can choose $X^\circ\subseteq X$ such that $\codim(X\setminus X^\circ)\geqslant 2$.

    Now we assume to the contrary that $\varphi$ is not the MRC fibration.  Let $g:X\dashrightarrow Y$ be the MRC fibration. Then there is a natural factorization $$X\dashrightarrow T \dashrightarrow Y.$$ Set $Z=X\backslash X^\circ$. Then $Z$ has codimension at least $2$. Thus, for a general fiber $G$ of $g$,  $Z\cap G$ also has codimension at least $2$ in $G$. Note that $G$ is smooth and rationally connected. By \cite[II, Proposition 3.7]{Kollar1996}, there is a very free rational curve $C$ in $G$ which is contained in $G\backslash Z$.  Hence $C$ is also a curve contained in $X^\circ$.  Denote by $r\colon \bbP^1\rightarrow \varphi^\circ(C')$ the normalization, and by $W$ the fiber product $X^\circ\times_{\bbP^1} T^\circ$. Then the natural morphism $\pi\colon W\rightarrow \bbP^1$ is a $\mathbb{P}^d$-bundle. This implies that there exists a vector bundle $\sE$ on $\bbP^1$ such that $W=\bbP(\sE)$ as the Brauer group of $\bbP^1$ is trivial. Moreover, since $\sF|_{X^\circ}$ is contained in $T_{X^\circ/T^\circ}$, the pullback $\sF'$ of $\sF$ on $W$ is a strictly nef locally free subsheaf of $T_{W/\bbP^1}$.  This contradicts with Corollary \ref{cor:proj-vector-bundle-p1-base}.
\end{proof}

\begin{lemma}\label{lemma:unsplit}
    Let $X$ be an $n$-dimensional projective manifold, and let $\cV$ be a covering family of minimal rational curves on $X$. If $T_X$ contains a strictly nef locally free subsheaf $\sF$, then $\cV$ is unsplit.
\end{lemma}

\begin{proof}
    If $\sF$ is a line bundle, this is a consequence of \cite[Corollaire]{Druel2004}. Therefore, we may assume $r=\rk \sF \geqslant 2$. Let $[C]\in \cV$ be a general member  and let $f\colon \bbP^1\rightarrow X$ be the morphism induced by the normalization of $C$. Then $f^*\sF$ is a subsheaf of $f^*T_X$. Since $\cV$ is a minimal family, there is some $d\geqslant r-1$ such that
    \[f^*T_X \cong \sO_{\bbP^1}(2)\oplus\sO_{\bbP^1}(1)^{\oplus d}\oplus\sO_{\bbP^1}^{\oplus (n-d-1)}.\]
    In particular, we have $c_1(\sF)\cdot C\leqslant r+1$.

    Now let $B=\sum a_i B_i$ be a $1$-cycle obtained as the limit of cycles in $\cV$ with $B_i$ irreducible and reduced.
    Then each $B_i$ is a rational curve. Since $\sF$ is strictly nef, it follows that  $c_1(\sF)\cdot B_i\geqslant r$ for all $i$.
    If $B$ is not reduced or not irreducible, then we have  $$c_1(\sF)\cdot B\geqslant 2r >r+1 \geqslant c_1(\sF)\cdot C.$$  This is a contradiction  and  hence $\cV$ is unsplit.
\end{proof}

\vskip 2\baselineskip

\section{Degeneration of $\bbP^d$}


\label{section:degeneration-P^d}

In order to prove Theorem \ref{thm:main-theorem}, we need to extend the projective bundle structure  $\varphi^\circ \colon X^\circ \to T^\circ$  obtained in Proposition \ref{MRC-fibration} from the MRC-fibration  to the whole manifold $X$.
To this end, we  study degenerations of  projective spaces.   Such  problems  have  been investigated in literatures, and we refer to \cite{Fujita1987,HoeringNovelli2013,AraujoDruel2014} and the references therein. Building on the work of Cho-Miyaoka-Shepherd-Barron \cite{ChoMiyaokaShepherd-Barron2002} (see also \cite{Kebekus2002}) and Koll\'ar \cite{Kollar2011}, we have  the following result which is essentially proved in \cite{HoeringNovelli2013} and \cite{AraujoDruel2014}. 

\begin{prop}
	\label{PROP:HN-degeneration-of-projective-spaces}
	Let  $\varphi\colon X\rightarrow T$ be an equidimensional fibration  between quasi-projective varieties  whose general fibers are isomorphic to $\bbP^d$. Assume that $T$ is normal and there exists a line bundle $\sL$ on $X$, whose restrictions on general fibers of $\varphi$ are isomorphic to $\sO_{\mathbb{P}^d}(e)$,  such that
	\begin{center}
		$c_1(\sL)\cdot C > \frac{1}{2} e$   for any  $\varphi$-exceptional rational curve  $C\subseteq  X$.
	\end{center}
	Then all the fibers of $\varphi$ are irreducible and generically reduced, and the normalization of any fiber is isomorphic to  $\bbP^d$.
	If we assume in addition that $X$ is normal, then   $\varphi\colon X\rightarrow T$ is a $\bbP^d$-bundle.
\end{prop}

\begin{proof}
	Let $\cH\subseteq \RC{X/T}$ be the unique irreducible component such that a general point corresponds to a line contained in the general fibers of $\varphi$. Let $[l]$ be a line contained in $\cH$. By assumption, we have
	\[
	c_1(\sL)\cdot C> \frac{1}{2}e = \frac{1}{2} c_1(\sL)\cdot l
	\]
	for any rational curve $C$ contracted by $\varphi$.
	Thus   $\cH$ is actually proper over $T$.
	In particular, we can apply the same argument as in \cite[p.222, Proof of Propositon 3.1]{HoeringNovelli2013} to show that the normalization of any irreducible component of any fiber of $\varphi$ is isomorphic to $\bbP^d$.

	Since $T$ is normal, by \cite[I, Definition 3.10 and Theorem 3.17]{Kollar1996}, $\varphi\colon X\rightarrow T$ is a well-defined family of $d$-dimensional proper algebraic cycles over $T$.
	In particular, the degree of the fibers with respect to $\sL$ is constant (see \cite[I, Lemma 3.17.1]{Kollar1996}).
	Then the computation   in \cite[p.223, Proof of Proposition 3.1]{HoeringNovelli2013} shows that the fibers of $\varphi$ are reduced and irreducible.
	Moreover, the pullback of $\sL$  on the normalization of any fiber $F$ of $\varphi$ is isomorphic to $\sO_{\bbP^d}(e)$.
	Thus the Hilbert polynomials of the normalizations of the fibers are the same.
	Then we apply \cite[Theorem 12]{Kollar2011} and obtain that $\varphi\colon X\rightarrow T$ admits a simultaneous normalization, which is a finite  birational morphism $\eta\colon \widetilde{X}\rightarrow X$ such that  the morphism $\varphi\circ\eta\colon \widetilde{X}\rightarrow T$ is flat with normal fibers.
	
	If we assume further that $X$ is normal, then Zariski's main theorem \cite[V, Theorem 5.2]{Hartshorne1977} implies that $\eta$ is an isomorphism. As a consequence  $\varphi\colon X\rightarrow T$ is a smooth morphism.
\end{proof}

We have the following corollary.

\begin{cor}\label{cor-equidim-implies-bundle}
	Let $\varphi \colon X\to T$ be an equidimensional Mori fibration from a smooth projective variety $X$ to a normal variety $T$. Assume that general fibers of $\varphi$ are isomorphic to $\bbP^d$ for some $d>0$. Assume further that there is a vector bundle $\sF$ of rank $r\geqslant 2$ whose restriction on every fiber is strictly nef. Then $\varphi$ is a $\mathbb{P}^d$-bundle between smooth varieties.
\end{cor}
\begin{proof}
	Let  $\sL\coloneqq \det(\sF)$. Then this line bundle satisfies the hypothesis of Proposition \ref{PROP:HN-degeneration-of-projective-spaces}. 
	Therefore $\varphi\colon X\rightarrow T$ is   a $\bbP^d$-bundle. Since both $X$ and $\varphi$ are smooth, it follows that $T$ is smooth.
\end{proof}

In order to apply  previous results, we need to ensure the equidimensionality of fibrations. Therefore, in the  remainder of this section,  we provide some criteria  for equidimensionality.

\begin{prop}[Criterion for equidimensionality I]
	\label{prop:equidimensionality-criterion-I}
	Assume that there is a commutative diagram  of fibrations of normal projective varieties
	\[
	\begin{tikzcd}[column sep=large, row sep=large]
	\overline{X}\dar[swap]{\overline{\varphi}} \rar{\rho} &  X\dar{{\varphi}} \\
	\overline{T} \rar{\gamma}& T
	\end{tikzcd}
	\]
	such that
	\begin{enumerate}
		\item $T$ has only klt singularities,
		\item general fibers of $\varphi$ are isomorphic to $\bbP^d$,
		\item $\overline T$ is smooth and the fibers of $\gamma$ over an open subset $T^\circ$ of $T$ with $\codim(T\setminus T^\circ)\geq 2$ are simply connected,
		\item  $\overline{\varphi}\colon \overline{X}\rightarrow \overline{T}$ is a flat $\mathbb{P}^d$-bundle given by a representation $\pi_1(\overline{T})\rightarrow \PGL_{d+1}(\bbC)$,
		\item there is a strictly nef vector bundle $\sF$ on $X$ with a surjective morphism $\rho^*\sF\rightarrow T_{\overline{X}/\overline{T}}$.
	\end{enumerate}
	Then $\varphi\colon X\rightarrow T$ is equidimensional.
\end{prop}

\begin{proof}
	Let $C\subseteq \overline{T}$ be a complete curve contracted by $\gamma$. Thanks to Lemma \ref{lemma:trivial-bundle-finite-change}, there exists a smooth curve $C'$ with a finite surjective morphism $n\colon C'\rightarrow C$ such that the fiber product $\overline{X}_{C'}\coloneqq \overline{X}\times_{\overline{T}} C'$ is isomorphic to   $C'\times \bbP^d$ as $\bbP^d$-bundles over $C'$.
	We denote by $$p_1\colon \overline{X}_{C'} \rightarrow C' \mbox{ and } \nu\colon \overline{X}_{C'} \to \overline{X}$$ the natural morphisms and by  \[
	p_2\colon \overline{X}_{C'}  \rightarrow \bbP^d
	\] the morphism induced by the natural projection from $C'\times \bbP^d$ to $\bbP^d$.
	
	Let $T_{\overline{X}_{C'}/C'}$ be the relative tangent bundle of $p_1$. Then its restriction on every fiber of $p_2$ is isomorphic to a trivial vector bundle.   On the other hand, the surjective morphism $\rho^*\sF\rightarrow T_{\overline{X}/\overline{T}}$  induces a surjective morphism $\nu^*\rho^*\sF\rightarrow T_{\overline{X}_{C'}/C'}$.  As a consequence, the restriction of $\nu^*\rho^*\sF$ on any fiber of $p_2$ is not strictly nef. This implies that the fibers of $p_2$ are contracted by the composition
	\[
	p\colon \overline{X}_{C'}\xrightarrow{\nu} \overline{X} \xrightarrow{\rho} X.
	\]
	By rigidity lemma, the morphism $p\colon \overline{X}_{C'}\rightarrow X$ factors through $p_2\colon \overline{X}_{C'}\rightarrow \bbP^d$. In particular, the images of all fibers of $\overline{\varphi}$ over $C$ in $X$ under $\rho$ coincide.
	
	Let $t\in T$ be an arbitrary point. Since the fiber $\overline{T}_t:= \gamma^{-1}(t)$ is connected, the previous paragraph implies that the images of all fibers of $\overline{\varphi}$ over $\overline{T}_t$ in $X$ under $\rho$ coincide.
	It follows that $\dim X_t\leqslant d$, where $X_t$ is the fiber of $X$ over $t$. By semicontinuity, we obtain that $\varphi$ is equidimensional.
\end{proof}

In the situation of Proposition \ref{MRC-fibration}, we can use the criterion in Proposition \ref{prop:equidimensionality-criterion-I} to deal with the case where the restriction $\sF\vert_F$ is isomorphic to $T_F$ for $F$ being a  general fiber of the MRC fibration. When $\sF|_F$ is isomorphic to the direct sum of copies of $\sO_F(1)$,  we need another more detailed treatment.

\begin{prop}
	\label{prop:equidimensionality-criterion II}
	Let $\varphi\colon X\rightarrow T$ be a surjective morphism  from a smooth projective variety $X$  to a normal projective  variety $T$, and let  $P\subseteq X$ be a  subvariety   such that the induced morphism $\varphi_P\coloneqq \varphi\vert_P\colon P\rightarrow T$ is surjective.
	If  at a point $x\in P$, the fiber $P_{\varphi(x)}$ is smooth of  dimension $\dim P-\dim T$, then the fiber $X_{\varphi(x)}$  is smooth  and of dimension $\dim X-\dim T$ at $x$.
\end{prop}

\begin{proof}
	Set $\dim X = n$, $\dim T=m$ and $\dim P = r$.  By assumption, we see that the relative differential sheaf $\Omega_{P/T}$ on $P$ is locally free of rank $r-m$ around $x$. Consider the following commutative diagram, with exact rows and columns,
	\[
	\begin{tikzcd}[column sep=large, row sep=large]
	\varphi^*\Omega_T\vert_{P}\rar \ar[d,equal]  &  \Omega_X\vert_{P}\dar\rar  &  \Omega_{X/T}\vert_{P}\rar\dar   & 0 \\
	\varphi_P^*\Omega_T\rar& \Omega_{P}\rar\dar    &  \Omega_{P/T}\ar[r]\ar[d]   & 0\\
	& 0 & 0 &
	\end{tikzcd}
	\]
	Let $\bbK(x)$ be the residue field of $P$ at $x$. Tensoring the diagram above with $\bbK(x)$, we have the following diagram, with exact rows and columns,
	\[
	\begin{tikzcd}[column sep=large, row sep=large]
	\varphi^*\Omega_T\vert_{P}\otimes\bbK(x)\rar{h}\ar[d,equal]  &  \Omega_X\vert_{P}\otimes\bbK(x)\dar\rar &  \Omega_{X/T}\vert_{P}\otimes\bbK(x)\rar \dar  & 0 \\
	\varphi_P^*\Omega_T\otimes\bbK(x)\rar{\overline{h}}& \Omega_{P}\otimes\bbK(x)\rar\dar     &  \Omega_{P/T}\otimes\bbK(x)\rar \dar  & 0\\
	& 0 & 0 &
	\end{tikzcd}
	\]
	The second row of the last diagram shows that
	\begin{align*}
	\dim_{\bbK(x)}(\im(\overline{h})) & = \dim_{\bbK(x)}(\Omega_P\otimes\bbK(x))-\left(r-m\right)\\
	& \geqslant r-\left(r-m\right)\\
	& = m.
	\end{align*}
	Since $\dim_{\bbK(x)}(\im(h))\geqslant \dim_{\bbK(x)}(\im(\overline{h}))$ and $X$ is smooth, the first row of the last diagram   implies that
	\begin{equation*}
	\label{equation:dimension-fibers-cotangents}
	\dim_{\bbK(x)}\left(\Omega_{X/T}\vert_{P}\otimes\bbK(x)\right)= n-\dim(\im(h))\leqslant n-m.
	\end{equation*}
	We note that $\Omega_{X/T}\vert_{P}\otimes\bbK(x) \cong \Omega_{X_{\varphi(x)}}\otimes\bbK(x)$, where $X_{\varphi(x)}$ is the fiber of $\varphi$ over $\varphi(x)$. Then we have
	\[
	n-m\leqslant \dim_x X_{\varphi(x)}  \leqslant \dim_{\bbK(x)} \Omega_{X_{\varphi(x)}}\otimes\bbK(x) \leqslant n-m
	\]
	This shows that $\Omega_{X/T}$ has rank $n-m$ around $x$. In particular, $\Omega_{X/T}$ is locally free around $x$ by Nakayama's lemma. Hence, $X_{\varphi(x)}$ is smooth at $x$ and has dimension $n-m$ at $x$.
\end{proof}

As an  application, we obtain  the following criterion for equidimensionality.

\begin{cor}[Criterion for equidimensionality II]
	\label{cor:equi-dim}
	Let $\varphi\colon X\rightarrow T$ be a surjective morphism  from a smooth projective  variety $X$  to a normal projective variety $T$.   Assume that there exists a (reduced and  irreducible) subvariety $P\subseteq X$  such that the induced morphism $\varphi_P\coloneqq \varphi\vert_P\colon P\rightarrow T$ is surjective and equidimensional with irreducible and generically reduced fibers. Let $t\in T$ be a point. Assume in addition that every component of the fiber $X_t$ contains the fiber $P_t$. Then $\varphi$ is equidimensional around $X_t$.
\end{cor}

\begin{proof}
	Since $\varphi_P$ is  equidimensional with irreducible and generically reduced fibers, there is a point $x\in P$ lying over $t$ such that the fiber $P_{t}$ is smooth at $x$. By Proposition \ref{prop:equidimensionality-criterion II}, $\varphi$ is equidimensional around $x$. Let $F$ be an irreducible component of $X_t$. Since $F$ contains $P_t$, we obtain that $x\in F$ and consequently $\dim F = \dim X -\dim T$.
\end{proof}


\vskip 2\baselineskip

\section{Proof of the projective bundle structure}
\label{section:bundle structure}

In this section, we   prove the   projective bundle structure in  Theorem \ref{thm:main-theorem}.  Actually, we will prove the following refined statement (compare it with Theorem \ref{Proj-bundle}).

\begin{thm}
	\label{cor:main:part1}
	Let $X$ be a  complex projective manifold. Assume that the tangent bundle $T_X$ contains a locally free strictly nef subsheaf $\sF$ of rank $r>0$.  Then $X$ admits a $\bbP^d$-bundle structure $\varphi\colon X\rightarrow T$ over a projective manifold $T$ for some $d\geq r$.
	Moreover, if $\dim T>0$, then exactly one of the following assertions holds.
	\begin{enumerate}
		\item Either $d=r\geq 1$, $\sF \cong T_{X/T}$ and $X$ is isomorphic to a flat projective bundle over $T$,
		\item or $r\geq 2$, $\sF$ is a numerically projectively flat vector bundle such that its restriction  on every fiber of $\varphi$ is  isomorphic to  $\sO_{\bbP^d}(1)^{\oplus r}$, and there exists a flat $\bbP^{r-1}$-subbundle $Q\rightarrow T$ of $X$ with a surjection $\sF\vert_Q\rightarrow T_{Q/T}$.  In particular, the relative tangent bundle $T_{Q/T}$ is strictly nef.
	\end{enumerate}
\end{thm}

\subsection{Setup}
\label{section:setup}

For the proof of Theorem \ref{cor:main:part1}, we  discuss two different cases, and each case consists  of several steps. For simplicity,  we  first establish some common setup in this subsection. 

Let $X$ be a projective manifold of dimension $n$, and $\sF \subseteq T_X$ a strictly nef locally free subsheaf of rank $r$. Then Proposition \ref{MRC-fibration} shows that there is an open subset $X^\circ$ of $X$ whose complement has codimension at least two such that there is a $\mathbb{P}^d$-bundle structure
$$
\varphi^\circ \colon X^\circ \to T^\circ
$$
and $\sF|_{X^\circ} \subseteq T_{X^\circ/T^\circ}$.  We recall that, by Lemma \ref{Pn}, for a general fiber $F$ of $\varphi^\circ$, the restriction $\sF|_F$ is either $T_F$ or isomorphic to $\sO_{F}(1)^{\oplus r}$. These two cases will be studied separately in Section \ref{subsection:F=T_F} and Section \ref{subsection:F=O(1)r}.  

The crucial part for Theorem \ref{cor:main:part1} is to prove the following result. 

\begin{thm}
	\label{lemma:Core-Technical-Theorem} Let $X$ be a projective manifold such that $T_X$ contains a locally
	free strictly nef  subsheaf.
	Then there exists an equidimensional Mori contraction $\varphi\colon X\rightarrow T$ which is also the MRC fibration of $X$.
\end{thm}

\noindent Indeed, by using Theorem \ref{lemma:Core-Technical-Theorem} and  Corollary \ref{cor-equidim-implies-bundle}, one can deduce  that   $\varphi\colon X\to T$ is a projective bundle, which gives the first part of Theorem 	\ref{cor:main:part1}. Finally we can finish the proof of Theorem \ref{cor:main:part1} by applying Theorem \ref{Proj-bundle}.
To show  Theorem 	\ref{lemma:Core-Technical-Theorem}, we analyze as follows. Let $T'$ be the normalization of the closure of $T^\circ$ in $\Chow{X}$, and let $X'$ be the normalization of the universal family over $T'$. We claim the following statement.

\begin{lemma}\label{lemma:bundle}
	The induced morphism $\varphi'\colon X'\rightarrow T'$ is a $\bbP^d$-bundle.
\end{lemma}

\begin{proof}
	We first assume that $\sF$ has rank $r\geqslant 2$.
	Denote by $e\colon X'\rightarrow X$ the evaluation morphism. Then the restrictions of  $e$ on fibers of $\varphi'$ are finite morphisms. In particular, the pullback $\sF'=e^*\sF$ is strictly nef when restricted on each fiber of $\varphi'$.
	Thus for a general fiber $F'$ of $\varphi'$, $\det(\sF')\vert_{F'}$ is isomorphic to either $\sO_{\bbP^d}(r)$ or $\sO_{\bbP^d}(r+1)$ after Lemma \ref{Pn}.
	If $C$ is a rational curve contained in a fiber  of $\varphi'$, then $\sF'\vert_C$ is strictly nef, and therefore ample. Hence  $c_1(\det(\sF))\cdot C\geqslant r$.
	By assumption, $r\geqslant \frac{1}{2}(r+1)$ and  we can apply Proposition \ref{PROP:HN-degeneration-of-projective-spaces} to conclude it.
	
	Now we assume that $\sF$ is a line bundle. We may assume further that $X$ is not isomorphic to a projective space. Then by \cite[Corollaire]{Druel2004},  there is a $\bbP^1$-bundle structure $\beta \colon X\to W$ with $\dim W > 0$ such that $\sF\cong T_{X/W}$. It then follows that $d=1$ and that $\beta$ is just the MRC fibration of $X$. Thus $X'\cong X$ and $T'\cong W$. This completes the proof of the lemma.
\end{proof}

Let $\overline{T}\rightarrow T'$ be a desingularization, and let $\overline{X}$ be the fiber product $X'\times_{T'}\overline{T}$.
Then $\overline{\varphi}\colon  \overline{X}\rightarrow \overline{T}$ is a $\bbP^d$-bundle by Lemma \ref{lemma:bundle}.
Let $Z$ be the fiber product $\overline{X}\times_{\overline{T}}\overline{X}$.
By Theorem \ref{thm:Brauer-group}, the induced morphism $Z\rightarrow \overline{X}$ is an insignificant projective bundle. In another word,  there exists a vector bundle $\sE$ of rank $d+1$ on $\overline{X}$ such that $Z=\bbP(\sE)$. Then we have the following commutative diagram, which will be frequently used throughout this section.

\begin{equation}\label{Key-diagram}
\begin{tikzcd}[column sep=large, row sep=large]
X^\circ \ar[d, "{\varphi^\circ}"]\ar[r] &
X & X' \ar[l,"e"]\ar[d, "{\varphi'}"]
& \overline{X} \ar[l,"{h}"] \arrow[bend right]{ll}[black,swap]{\rho}   \ar[d,"{\overline{\varphi}}"]    &
Z =\bbP(\sE) \ar[l,"{q}"]\ar[d,"{\pi}"]
\\
T^\circ & &
T' &
\overline{T} \ar[l] &
\overline{X}\ar[l,"{p=\overline{\varphi}}"]
\end{tikzcd}
\end{equation}

The first step  towards the proof of Theorem \ref{lemma:Core-Technical-Theorem}  is to show that the inclusion $\sF \hookrightarrow T_{X}$ induces an inclusion $\rho^*\sF \to T_{\overline{X}/\overline{T}}$. 
To achieve this, we will proceed as follows. 
Let $E_i$'s be the $\rho$-exceptional prime divisors. 
By shrinking $X^\circ$ if necessary, we may identify  $X^\circ$ with an open subset  $\overline{X}^\circ$ of $\overline{X}$. 
Then there are smallest integers $m_i$ such that the morphism $\sF|_{X^\circ} \to T_{X^\circ/T^\circ}$ extends to a morphism 
\begin{equation}
\rho^*\sF \to T_{\overline{X}/\overline{T}}\otimes \sO_{\overline{X}}(\sum m_iE_i).
\end{equation}
Our goal is then to prove that  $m_i \leqslant 0$ for all $i$.

We also have the following simple observation. By construction, $T^\circ$ can be identified with an open subset of $\overline{T}$ such that $\overline{\varphi}(E_i)$ is contained in $\overline{T}\setminus T^\circ$ for all $i$. As a consequence, there are prime divisors $Q_i$ in $\overline{T}$ such that $E_i=\overline{\varphi}^*Q_i$. In particular, we have
$$
q^*E_i=\pi^*E_i.
$$
Let $E_i'= q^*E_i=\pi^*E_i$  and let  $\sG=q^*\rho^*\sF$. We denote by $T_{\pi}$ the relative tangent bundle of $\pi\colon Z\rightarrow \overline{X}$.
Then the $m_i$ are also the  smallest integers such that there is a morphism 
\begin{equation}
\label{equation:lifing-morphism}
\Phi\colon \sG \to T_{\pi}\otimes \sO_{Z}(\sum m_iE_i'),
\end{equation}
which extends the following  morphism on $Z^\circ := \pi^{-1}(\overline{X}^\circ)$
\begin{equation*}
q^*\rho^*(\sF|_{X^\circ}) \hookrightarrow q^*\rho^*(T_{X^\circ/T^\circ})\cong T_{Z^\circ/\overline{X}^\circ}.
\end{equation*}
Moreover, $\Phi$ does not vanish in codimension one by the minimality of $m_i$.

\subsection{The case when $\sF|_F=T_F$}
\label{subsection:F=T_F}

In this subsection, we will  study the case when $\sF\vert_F=T_F$.  Note that we have $r=d$ in this case. Since the proof   is a bit involved, we  subdivide  it into four steps, given in Sections \ref{section:step-1-case-1}--\ref{section:step-4-case-1} below. We shall follow the notations  in Section \ref{section:setup}, especially those in the commutative diagram \eqref{Key-diagram}.

\subsubsection{Lifting the inclusion $\sF\hookrightarrow T_X$ to an isomorphism $\rho^*\sF\rightarrow T_{\overline{X}/\overline{T}}$}
\label{section:step-1-case-1}

\begin{claim}\label{claim:pullback-sheaf=relative-tangent}
	The injection $\sF|_{X^\circ} \hookrightarrow T_{X^\circ/T^\circ}$ is an isomorphism.
\end{claim}

\begin{proof}
	Let $t\in T^\circ$ be an arbitrary point and let $x$ be a point in $\varphi^{\circ-1}(t)$. Since the complement of $X^\circ$ in $X$ has codimension at least 2, we may choose a general complete curve $C$ in $X$ passing through $x$ such that $B\coloneqq\varphi^\circ(C)$ is a complete curve in $T^\circ$. Let $n\colon B'\rightarrow B$ be the normalization of $B$ and let $Y$ be the fiber product $X^\circ\times_{B} B'$. Then $p\colon Y\rightarrow B'$ is a $\bbP^d$-bundle. Since the Brauer group of $B'$ is trivial, there exists a vector bundle $\sV$ over $B'$ such that $Y\cong \bbP(\sV)$.
	\[\begin{tikzcd}[column sep=large, row sep=large]
	Y=\bbP(\sV) \dar[swap]{p}   \arrow{rr}[swap]{\nu}     &      &  X^\circ\dar{\varphi^\circ}  \arrow[r,hookrightarrow] & X\\
	B'  \arrow[r,hookrightarrow]        &    B  \rar    & T^\circ &
	\end{tikzcd}\]
	Since $C$ is in general position, the induced morphism $\nu^*\sF\rightarrow T_{Y/B'}$ is still injective such that $(\nu^*\sF)\vert_G=T_G$ for general fibers $G$ of $p$. As $\nu$ is finite, $\nu^*\sF$ is strictly nef. Applying Theorem \ref{Proj-bundle} to $p\colon Y\rightarrow B'$ shows that $\nu^*\sF=T_{Y/B'}$. As $t$ is arbitrary, by pushing-forward, we obtain that $\sF\vert_{X^\circ}\rightarrow T_{X^\circ/T^\circ}$ is an isomorphism.
\end{proof}

\begin{claim}\label{claim:pullback-sheaf-in-tangent-1}
	The injection $\sF \hookrightarrow T_X$ induces an isomorphism   $\rho^*\sF \to T_{\overline{X}/\overline{T}}$.
\end{claim}

\begin{proof}
	Let $\sG=q^*\rho^*\sF$. As explained at the end of Section \ref{section:setup}, we have the following commutative diagram
	
	\[\begin{tikzcd}[column sep=large, row sep=large]
	X^\circ \dar[swap]{\varphi^\circ}\rar &
	X & \overline{X} \lar[swap]{\rho} \dar[swap]{\overline{\varphi}} &
	Z=\bbP(\sE)\lar[swap]{q}\dar{\pi} \\
	T^\circ & &
	\overline{T} &
	\overline{X}\lar{p=\overline{\varphi}}
	\end{tikzcd}
	\]
	and an induced morphism (see \eqref{equation:lifing-morphism} for details)
	$$   
	\Phi\colon \sG \to T_{\pi}\otimes \sO_{Z}(\sum m_iE_i'),
	$$
	which does not vanish in codimension one and extends the natural isomorphism 
	\[
	(q^*\rho^*\sF)\vert_{X^\circ}\rightarrow q^*\rho^*(T_{X^\circ/T^\circ})\cong T_{Z^\circ/\overline{X}^\circ}.
	\]
	\smallskip

	\textit{1st Step. $\sE$ is numerically projectively flat.} We consider the line bundle $\sL = \mathrm{det}\, \sG$.
	Its restriction on any fiber of $\pi$ is isomorphic to $\sO_{\mathbb{P}^d}(d+1)$.
	Thus there is a line bundle $\sH$ on $\overline{X}$ such that
	$$
	\sL \cong \sO_{\mathbb{P}(\sE)}(d+1) \otimes \pi^*\mathrm{det}\, \sE^* \otimes \pi^*\sH.
	$$
	We note that $\sG|_{Z^\circ} \cong T_{Z^\circ/\overline{X}^\circ}$ by Claim \ref{claim:pullback-sheaf=relative-tangent}.
	This implies that $\sH|_{\overline{X}^\circ}\cong \sO_{\overline{X}^\circ}$.
	Therefore, there is a  $\mathbb{Q}$-divisor class $\delta$, supported in the $\rho$-exceptional locus, such that $$c_1(\sH) = (d+1) \delta.$$
	Since $\sL$ is nef, by definition the $\mathbb{Q}$-twisted vector bundle $\sE \hspace{-0.8ex}<\hspace{-0.8ex} \delta - \mu(\sE) \hspace{-0.8ex}> $ is  nef, where $\mu(\sE)$ is the average first Chern class of $\sE$.
	By taking the first Chern class, we see that $\delta$ is nef. Since $\delta$ is supported in the $\rho$-exceptional locus, by the negativity lemma, we obtain that $\delta=0$.
	Thus $$\sO_{\bbP(\sE)}(d+1)\otimes \pi^*\det\sE^*$$ is nef and consequently $\sE$ is numerically projectively flat.\\
	\smallskip	

	\textit{2nd Step.   $m_i\leqslant 0$ for all $i$.} Assume by contradiction that it is not the case.
	By Lemma \ref{lemma:negativity-lemma-line-bundle}, there is a family $\{C_\gamma\}_{\gamma\in \Gamma}$ of complete $\rho$-exceptional curves  such that $C_\gamma\cdot \sum m_iE_i<0$ for all $\gamma\in \Gamma$.
	Let $C'$ be a general member of these curves.
	
	Since $\sE$ is numerically projectively flat, $\sE$ is isomorphic to a projectively flat vector bundle by Theorem \ref{thm:num-proj-flat=proj-flat}. In particular, $\pi\colon Z\rightarrow X$ is isomorphic to a flat projective bundle over $\overline{X}$. Applying Lemma \ref{lemma:trivial-bundle-finite-change} to $\rho\colon \overline{X}\rightarrow X$ and $\pi\colon Z\rightarrow \overline{X}$, we deduce that there is a smooth complete curve $C$, which is finite over $C'$, such that  $Z\times_{\overline{X}} C$ is isomorphic to $\mathbb{P}^d\times C$ as $\bbP^d$-bundles over $C$.
	\[
	\begin{tikzcd}[column sep=large, row sep=large]
	D \arrow[bend left]{rrr}[swap]{\nu}\arrow[r,hookrightarrow] &  \mathbb{P}^d\times C \cong Z\times_{\overline{X}} C \arrow[rr] \dar & & Z \dar{\pi}\\
	& C  \rar & C' \arrow[r,hookrightarrow] &  \overline{X}
	\end{tikzcd}
	\]
	Let $D$ be a general fiber of the natural projection $\mathbb{P}^d\times C\to \mathbb{P}^d$. We denote by $\nu\colon D \to Z$ the natural morphism. Since the morphism $\Phi\colon\sG\rightarrow T_{\pi}\otimes \sO_{Z}(\sum m_i E_i')$ does not vanish in codimension one,  by general choices of $C'$ and $D$, we may assume that   the morphism $$\sG|_{\nu(D)}  \to (T_{\pi} \otimes \sO_{Z}(\sum m_iE_i'))|_{\nu(D)}$$  is not zero.
	Since $\nu^*T_{\pi} \cong T_{(\mathbb{P}^d\times C)/C}|_D$ is trivial and  $\sG$ is nef, we obtain that $\nu^* \sO_{Z}(\sum m_iE_i')$ is pseudoeffective. This contradicts to the fact that 
	$$
	\nu(D)\cdot \sum m_i E_i'=\deg(\pi\vert_{\nu(D)}) C'\cdot \sum m_iE_i <0.
	$$
	Hence, we have $m_i\leq 0$ for all $i$ and there is  an induced injective morphism 
	\[
	\rho^*\sF\rightarrow T_{\overline{X}/\overline{T}}.
	\]
	\smallskip

	\textit{3rd Step. $\rho^*\sF\rightarrow T_{\overline{X}/\overline{T}}$ is an isomorphism.} By Theorem \ref{Proj-bundle}, the induced morphism $q^*\rho^*\sF\rightarrow T_{\pi}$ is an isomorphism.  Hence   $\rho^*\sF \rightarrow T_{\overline{X}/\overline{T}}$ is   an isomorphism  by taking pushforward.
\end{proof}

\subsubsection{Regularity of the MRC fibration}

The existence of $\varphi\colon X\rightarrow T$ is a direct consequence of the following claim, which asserts that $X$ admits only one elementary contraction.

\begin{claim}
	\label{claim:extremal-ray-F=T_F}
	Let $C$ be a  rational curve in $X$. Then $C$ is numerically proportional to a line $l$ contained in a fiber  of $\varphi^\circ$.
\end{claim}

\begin{proof}
	By Lemma \ref{lemma:rational-section-birational-morphism}, there exists a curve $C'$ contained in $X'$ such that $e\vert_{C'}\colon C'\rightarrow C$ is birational. In particular, $C'$ is a rational curve.
	\begin{equation*}
	\begin{tikzcd}[column sep=large, row sep=large]
	X^\circ \dar{\varphi^\circ}\rar &
	X & X' \lar{e} \dar[swap]{\varphi'}
	& \overline{X} \lar{h} \arrow[bend right]{ll}[swap]{\rho} \dar{\overline{\varphi}}
	\\
	T^\circ & &
	T' &
	\overline{T} \lar
	\end{tikzcd}
	\end{equation*}
	Note that $\varphi'\colon X'\rightarrow T'$ is a $\bbP^d$-bundle by Lemma \ref{lemma:bundle}, thus any complete curve contained in a fiber  of $\varphi'$ is numerically proportional to a line contained in a fiber of $\varphi'$.
	In particular, if $C'$ is contained in a fiber  of $\varphi'$, by Lemma \ref{lemma:numerical-proportional-preseved}, $C'$ is numerically proportional to a line contained in a fiber  of $\varphi^\circ$.
	
	Now we assume  that $\varphi'(C')=B$ is a curve.
	Let $\bbP^1\rightarrow B$ be the normalization. Denote by $X'_B$ the fiber product $X'\times_{T'}\bbP^1$ with induced morphism $\nu\colon X'_B\rightarrow X'$.
	Since $p_1\colon X'_B\rightarrow \bbP^1$ is a $\bbP^d$-bundle and the Brauer group of $\bbP^1$ is trivial,  there exists a vector bundle $\sV$  on $\bbP^1$ such that $X'_B= \bbP(\sV)$.
	Moreover, since  $\rho^*\sF \cong T_{\overline{X}/\overline{T}}$ by Claim \ref{claim:pullback-sheaf-in-tangent-1} and since $T_{\overline{X}/\overline{T}}\cong  h^*T_{X'/T'}$, we get
	\[
	\rho^*\sF = h^*e^*\sF  \cong T_{\overline{X}/\overline{T}}  \cong h^*T_{X'/T'}.
	\]
	Hence  $e^*\sF \cong T_{X'/T'}$. In particular, it yields
	\[
	\nu^*e^*\sF \cong  T_{\overline{X}_B/\bbP^1}.
	\]
	Then   Theorem \ref{Proj-bundle} shows that $\sV$ is numerically projectively flat. As a consequence, we obtain that $X'_{B}\cong  \bbP^1\times \bbP^d$. Let $p_2\colon X'_{B} \to \bbP^d$ be the morphism induced by the projection $\bbP^1\times \bbP^d\rightarrow \bbP^d$ and $f\colon X'_{B} \to X$ the composition of
	\[
	X'_B\xrightarrow{\nu} X'\xrightarrow{e} X.
	\]
	Since $\nu^*e^*\sF \cong  T_{\overline{X}_B/\bbP^1}$ is trivial on the fibers of $p_2$  and since $\sF$ is  strictly nef, the fibers of $p_2$ are all contracted by $f$.
	Hence, by rigidity lemma, the morphism $f\colon {X}'_B\rightarrow X$ factors through $p_2$. As a consequence, every point in $B\subseteq T'$ corresponds to the same cycle in $X$. This contradicts to the definition of $\Chow{X}$. Hence $C'$ is always contracted by $\varphi'$, and we complete the proof of the claim.
\end{proof}

\subsubsection{Proof of Theorem \ref{lemma:Core-Technical-Theorem} in the case when $\sF\vert_F=T_F$}

\label{section:step-3-case-1}

By Claim \ref{claim:extremal-ray-F=T_F}, there exists a Mori contraction $\varphi\colon X\rightarrow T$ extending the   fibration $\varphi^\circ\colon X^\circ\rightarrow T^\circ$. By replacing $\overline{T}$   with a common resolution of $T'$ and $T$, we have the following commutative diagram:
\[
\begin{tikzcd}[column sep=large, row sep=large]
X^\circ \dar[swap]{\varphi^\circ}\rar &
X\dar[swap]{\varphi} & X' \lar{e}\dar[swap]{\varphi'}
& \overline{X} \lar{h} \arrow[bend right]{ll}[swap]{\rho} \dar{\overline{\varphi}}    &
Z = \bbP(\sE)\lar[swap]{q}\dar{\pi}
\\
T^\circ & T &
T' &
\overline{T} \lar \arrow[bend left]{ll}{\gamma} &
\overline{X}\lar{p=\overline{\varphi}}
\end{tikzcd}
\]
By Claim \ref{claim:pullback-sheaf-in-tangent-1}, $\sF\hookrightarrow T_X$ induces  an isomorphism $\sG = q^*\rho^*\sF \to T_{Z/\overline{X}}$. By using
Theorem \ref{Proj-bundle} and Theorem \ref{thm:num-proj-flat=proj-flat},  we deduce that $\sE$ is isomorphic to a  projectively flat vector bundle.
Hence   $\pi$ is isomorphic to a  $\bbP^d$-bundle structure given by a representation $\pi_1(\overline{X})\rightarrow \PGL_{d+1}(\bbC)$.
Since $p\colon \overline{X}\rightarrow \overline{T}$ has simply connected fibers,  we deduce that  $\overline{\varphi}\colon \overline{X}\rightarrow \overline{T}$ is also isomorphic to a flat  $\bbP^d$-bundle over $\overline{T}$ given by a representation   $\pi_1(\overline{T})\rightarrow \PGL_{d+1}(\bbC)$. We also note that $T$ has only klt singularities. Now one can derive the equidimensionality of $\varphi$ from Proposition \ref{prop:equidimensionality-criterion-I}. \hfill \qed

\subsubsection{Proof of Theorem \ref{cor:main:part1} in the case when $\sF\vert_F=T_F$}

\label{section:step-4-case-1}
We maintain the notations of Section \ref{section:setup}.  We first show the projective bundle structure on $X$. 
If $\sF$ is a line bundle, then we have a $\bbP^1$-bundle structure $\varphi\colon X\to T$ with $\sF\cong T_{X/T}$ by \cite[Corollaire]{Druel2004}. Thus we may assume that $r\geq 2$. By the corrsponding case in Theorem \ref{lemma:Core-Technical-Theorem}, there exists an equidimensional Mori contraction $\varphi\colon X\rightarrow T$ which is also the MRC fibration of $X$. Since $\sF$ is strictly nef of rank at least two, by Corollary \ref{cor-equidim-implies-bundle}, $\varphi$ is again a projective bundle.

Next we assume that $\dim T>0$. Since $\varphi\colon X\rightarrow T$ is a projective bundle structure between projective manifolds, we may identify $T$ with $\overline{T}$ and $X$ with $\overline{X}$. Then the fiber product $X\times_T X$ can be identified with $Z=\bbP(\sE)$. In particular, we have the following commutative diagram:
\[
\begin{tikzcd}[column sep=large, row sep=large]
X\dar[swap]{\varphi} & Z \lar{q}\dar[swap]{\pi} \\
T &
X\lar{p=\varphi}
\end{tikzcd}
\]
We have seen $\sF\cong T_{X/T}$ by Claim \ref{claim:pullback-sheaf-in-tangent-1}. Moreover,  from the first step of Claim \ref{claim:pullback-sheaf-in-tangent-1}, $\sE$ is isomorphic to a projectively flat vector bundle. Therefore, $Z$ is isomorphic to a flat projective bundle over $X$,  given by a representation of $\pi_1(X)$ in $\PGL_{d+1}(\bbC)$.
Since $X$ is a projective bundle over $T$, the fundamental group of $X$ is canonically isomorphic to that of $T$.
Hence we get an induced flat projective bundle $Q$ over $T$. Let $f\colon Z\rightarrow Q$ be the natural projection. Then an irreducible $C\subset Z$ is contracted by $q$ if and only if $C$ is contracted by $f$.  Thus, by rigidity lemma, one can easily derive that $X$ is isomorphic to $Q$ as projective bundles over $T$. 
\hfill \qed

\subsection{The case when $\sF|_F=\sO_{\bbP^d}(1)^{\oplus r}$}
\label{subsection:F=O(1)r}

In this subsection, we  study the case when $\sF\vert_F\cong  \sO_{\bbP^d}(1)^{\oplus r}$.   As in the previous subsection, the proof  is subdivided into four parts, given in Sections \ref{section:step-1-case-2}--\ref{section:step-4-case-2}. We still follow the  notations in Section \ref{section:setup}.

\subsubsection{Lifting the inclusion $\sF\hookrightarrow T_X$ to an inclusion $\rho^*\sF\hookrightarrow T_{\overline{X}/\overline{T}}$}

\label{section:step-1-case-2}

\begin{claim}\label{claim:pullback-saturated-subsheaf}
	Use notations as in Section \ref{section:setup} and identify $X^\circ$ with $\overline{X}^\circ$. Let $\sE^\circ$ be the restriction $\sE\vert_{X^\circ}$. Then there is a vector bundle $\sM^\circ$ on $X^\circ$  such that 
	$$
	p_1^*\sF \cong p_2^*\sM^\circ \otimes \sO_{\mathbb{P}(\sE^\circ)}(1),
	$$ 
	where $p_1$ and $p_2$ are the natural projection $\pi\vert_{Z^\circ}$ and $q\vert_{Z^\circ}$ respectively. Furthermore, the inclusion $\sF|_{X^\circ} \hookrightarrow T_{X^\circ/T^\circ}$ induces a subbundle structure $\sM^\circ \hookrightarrow  (\sE^\circ)^*$.
\end{claim}

\begin{proof}
	By construction, we have $Z^\circ=\bbP(\sE^\circ)$. Since relative tangent bundles commute  with base change, there is an induced  inclusion $$p_1^*\sF \hookrightarrow  T_{p_2},$$ where   $T_{p_2}$ is the relative tangent bundle of $p_2$.
	Moreover, for every fiber $G$ of $p_2$, we have $p_1^*\sF|_{G} \cong \sO_{\bbP^d}(1)^{\oplus r}$.
	Thus there is a vector bundle $\sM^\circ$ on $X^\circ$  such that 
	$$
	p_1^*\sF \cong p_2^*\sM^\circ \otimes \sO_{\mathbb{P}(\sE^\circ)}(1).
	$$
	Furthermore,  by Lemma \ref{lemma:factorization}, the inclusion $\sF \hookrightarrow T_{X^\circ/T^\circ}$ induces an inclusion $\sM^\circ \hookrightarrow (\sE^\circ)^*$.
	
	Next we show that $\sM^\circ$ is a subbundle of $(\sE^\circ)^*$. Let $x$ be an arbitrary point in $X^\circ$ and   $C$  a  general complete intersection curve in $X$ passing through $x$.
	We may choose $C$ so that $C\subseteq X^\circ$. Let  $Y = \bbP(\sE^\circ\vert_C)$.
	Then we have the following commutative diagram.
	\[\begin{tikzcd}[column sep=large, row sep=large]
	Y=\bbP(\sE^\circ\vert_C) \dar{p}   \rar      &    Z^\circ=\bbP(\sE^\circ)\rar{p_1}\dar[swap]{p_2}  &  X^\circ\dar{\varphi^\circ}  \arrow[r,hookrightarrow] & X\\
	C  \arrow[r,hookrightarrow]   &    X^\circ  \rar   & T^\circ &
	\end{tikzcd}
	\]
	Denote by $f \colon Y\to X$ the composition of the first row.
	Since $C$ is in general position, we still have an injective morphism $f^*\sF \to T_{Y/C}$.
	By applying Lemma \ref{lemma:factorization} again, we see that the induced morphism $\sM^\circ|_C \to (\sE^\circ)^*\vert_C$  is still injective. Then Theorem \ref{Proj-bundle}  implies that $\sM^\circ|_C$ is a subbundle of $(\sE^\circ)^*\vert_C$. Since $x$ is arbitrary, $\sM^\circ$ is a subbundle of $(\sE^\circ)^*$
\end{proof}

\begin{claim}\label{claim:pullback-sheaf-in-tangent}
	The injection $\sF \hookrightarrow T_X$ induces an injection  $\rho^*\sF \hookrightarrow T_{\overline{X}/\overline{T}}$.
	Moreover, its restriction on each fiber of $\overline{\varphi}$ is still injective.
\end{claim}

\begin{proof}
	Let $\sG=q^*\rho^*\sF$. As explained at the end of Section \ref{section:setup}, we have the following commutative diagram
	\[\begin{tikzcd}[column sep=large, row sep=large]
	X^\circ \dar{\varphi^\circ}\rar &
	X & \overline{X} \lar[swap]{\rho}\dar{\overline{\varphi}} &
	Z=\bbP(\sE)\lar[swap]{q}\dar{\pi} \\
	T^\circ & &
	\overline{T} &
	\overline{X}\lar{p=\overline{\varphi}}
	\end{tikzcd}\] 
	and an induced morphism (see \eqref{equation:lifing-morphism} )
	$$
	\Phi\colon \sG \to T_{\pi}\otimes \sO_{Z}(\sum m_iE_i')
	$$
	which does not vanish in codimension one. Alternatively, we have a morphism 
	$$
	\sG \otimes \pi^*\sO_{\overline{X}}(\sum -m_iE_i') \to T_{\pi}
	$$ 
	on $Z$, which is nonzero in codimension one.
	We note that the restriction of $\sG$ on every fiber of $\pi$ is  isomorphic to $\sO_{\mathbb{P}^d}(1)^{\oplus r}$ (see Lemma \ref{Pn}).
	Hence $$\sG  \cong  \pi^*\sM \otimes \sO_{\mathbb{P}(\sE)}(1)$$  for some vector bundle $\sM$ on $ \overline{X}$.
	Thanks to Lemma \ref{lemma:factorization}, we obtain an injective morphism $$\sM\otimes \sO_{\overline{X}}(\sum -m_iE_i) \to \sE^*$$ on $\overline{X}$
	which is nonzero in codimension one.  We denote by $\sQ$  the  saturation of $\sM\otimes \sO_{\overline{X}}(\sum -m_i E_i)$     in $\sE^*$. \\
	
	

	\textit{1st  Step. $\sQ$ is a numerically projectively flat vector bundle.}
	By Claim \ref{claim:pullback-saturated-subsheaf}, the injection $\sF \hookrightarrow T_X$ induces  a subbundle  structure $\sM|_{\overline{X}^\circ} \hookrightarrow \sE^*|_{\overline{X}^\circ}$.
	Thus $\sQ\vert_{\overline{X}^\circ} = \sM\vert_{\overline{X}^\circ}$ and there exists a $\bbQ$-divisor $\delta$ supported in the $\rho$-exceptional locus  such that 
	$$
	\mu(\sM)-\mu(\sQ) = \delta.
	$$	
	Since $\sG\cong \pi^*\sM\otimes\sO_{\bbP(\sE)}(1)$ is nef, so is its determinant 
	$$
	\det \sG \cong \pi^*\det \sM \otimes \sO_{\bbP(\sE)}(r).
	$$  
	In particular, $\sE\hspace{-0.8ex}<\hspace{-0.8ex}\mu(\sM)\hspace{-0.8ex}>$ is nef. Therefore,  from the generically surjective morphism $\sE \to \sQ^*$, we know that the $\mathbb{Q}$-twisted sheaf
	\[
	\sQ^*\hspace{-0.8ex}<\hspace{-0.8ex}\mu(\sM)\hspace{-0.8ex}>
	\]
	is almost nef. By taking the first Chern class, we obtain that $\delta$ is pseudoeffective. As $\delta$ is $\rho$-exceptional, it follows that $\delta$ is effective.

	Next we  show that $\delta=0$. Assume the  opposite.
	By Lemma \ref{lemma:negativity-lemma-line-bundle},  there is a family $\{C_\gamma\}_{\gamma\in \Gamma}$ of complete $\rho$-exceptional curves  such that $C_\gamma\cdot \delta<0$ for all $\gamma\in \Gamma$.
	Since $\sQ$ is locally free in codimension two and  it is a subbundle of $\sE^*$ in codimension one,  we may choose a general element $C'$ in $\{C_\gamma\}_{\gamma\in \Gamma}$ such that $\sQ$ is locally free along $C'$ and the induced morphism $\sE\vert_{C'}\rightarrow \sQ^*\vert_{C'}$ is generically surjective.
	This implies that   $\sQ^*\hspace{-0.8ex}<\hspace{-0.8ex}\mu(\sM)\hspace{-0.8ex}>\vert_{C'}$ is an almost nef vector bundle. By taking the determinant, we deduce that 
	$$
	C'\cdot \delta=C'\cdot (\mu(\sM) - \mu(\sQ))\geqslant 0,
	$$ 
	which yields a contradiction. Hence $\delta=0$ and we have $\mu(\sM)=\mu(\sQ)$.
	In particular, as $\sQ^*\hspace{-0.8ex}<\hspace{-0.8ex}\mu(\sM)\hspace{-0.8ex}>$ is almost nef, by Theorem \ref{Num-Projectily-flatness-criterion}, $\sQ^*$ is actually a numerically projectively flat vector bundle. So is $\sQ$ by Lemma \ref{lemma:prop-num-proj-flat}.\\
	
	\textit{2nd Step. $\sM$ is isomorphic to a projectively flat vector bundle.} For simplicity, we  will argue up to isomorphisms. We may assume that $\sQ$ is projectively flat by Theorem \ref{thm:num-proj-flat=proj-flat}.
	Since $$\sG|_{Z^\circ} \cong (\pi^*\sQ \otimes \sO_{\bbP(\sE)}(1))|_{Z^\circ},$$ we see that $\sG|_{Z^\circ}=(q^*\rho^*\sF)|_{Z^\circ}$ is projectively flat. Then $\rho^*\sF|_{\overline{X}^\circ}$ is projectively flat by Lemma \ref{lemma:proj-flat-vec-bundle-property2}. Therefore, $\sF|_{X^\circ}$ is projectively flat. By Lemma \ref{lemma:extension-proj-flat-vec-bundle}, we obtain that $\sF$ itself is projectively flat. Hence $\sG$ is projective flat and so is   $\sM$ by Lemma \ref{lemma:proj-flat-vec-bundle-property2}.\\

	\textit{3rd Step.  $m_i \leqslant 0$ for all $i$.}
	Assume the opposite. By Lemma \ref{lemma:negativity-lemma-line-bundle}, there is a family $\{C_\gamma\}_{\gamma\in \Gamma}$ of complete $\rho$-exceptional curves  such that $C_\gamma\cdot \sum m_iE_i<0$ for all $\gamma\in \Gamma$.
	Let $C'$ be a general member of these curves. Since $\sM$ is isomorphic to a projectively flat vector bundle, $\bbP(\sM)$ is isomorphic to a flat projective bundle. Now applying Lemma \ref{lemma:trivial-bundle-finite-change} to $p\colon \overline{X}\rightarrow \overline{T}$ and $\bbP(\sM)\rightarrow \overline{X}$ shows that   there is a smooth  curve $C$, finite over $C'$, such that  $\eta^*\sM \cong \sL^{\oplus r}$ for some line bundle $\sL$ on $C$, where $\eta\colon C\to \overline{X}$ is the natural morphism. Since $\sE\rightarrow \sM^*\otimes \sO_{\overline{X}}(\sum m_i E_i)$ does not vanish in codimension one, by general choice of $C'$, we may assume that the morphism $$\eta^*\sE \to \eta^*(\sM^* \otimes  \sO_{\overline{X}}(\sum m_i E_i))$$ is not identically zero.
	Hence we obtain a generically surjective morphism
	$$\eta^*\sE \to \sL^*\otimes \eta^*\sO_{\overline{X}}(\sum m_i E_i).$$
	Since the $\mathbb{Q}$-twisted sheaf $\sE\hspace{-0.8ex}<\hspace{-0.8ex}\mu(\sM)\hspace{-0.8ex}>$ is nef, it follows that $\eta^*\sE\otimes \sL$ is nef. Thus $\eta^*\sO_{\overline{X}}(\sum m_i E_i)$ is pseudoeffective on $C$. This contradicts to the fact that   $C'\cdot \sum m_iE_i<0$. As a consequence, we have  induced morphisms  $\rho^*\sF \to T_{\overline{X}/\overline{T}}$ and $\sM \hookrightarrow \sE^*$.\\
	
	\textit{4th Step. The restriction of $\rho^*\sF \to T_{\overline{X}/\overline{T}}$ on any fiber of $\overline{\varphi}$ is still injective.}  By Theorem \ref{Proj-bundle}, we see that  $\sM\hookrightarrow \sE^*$ is a subbundle.  The statement then follows.
\end{proof}

\subsubsection{Regularity of the MRC fibration}

Similarly to the case when $\sF\vert_F\cong T_F$, we prove the following claim.

\begin{claim}\label{claim:extremal-ray-F=O1}
	Let $C\subseteq X$ be a rational curve. Then $C$ is numerically proportional to a line contained in a fiber  of $\varphi^\circ$.
\end{claim}

\begin{proof}
	By Lemma \ref{lemma:rational-section-birational-morphism}, there exists a complete rational curve $C'\subseteq \overline{X}$  which is birational to $C$.  As in the proof of Claim \ref{claim:extremal-ray-F=T_F}, it is enough to consider the case when $C'$ is not contracted by $\overline{\varphi}$.
	
	Assume that $B=\overline{\varphi}(C')$ is a curve. Then $B$ is a rational curve.  Consider the normalization $\bbP^1 \to B$.
	Let $\overline{X}_B = \overline{X}\times_{\overline{T}} \bbP^1$  and $p_1\colon \overline{X}_B\rightarrow \bbP^1$ the natural projection.
	\[\begin{tikzcd}[column sep=large, row sep=large]
	X^\circ \dar[swap]{\varphi^\circ}\rar &
	X & \overline{X} \lar{\rho}\dar[swap]{\overline{\varphi}} & \overline{X}_B \cong\bbP(\sV)\dar{p_1}\lar{\nu}\arrow[bend right]{ll}[swap]{f} \rar{g} & Y \\
	T^\circ & &
	\overline{T} & \bbP^1\lar  &
	\end{tikzcd}\]
	Since $p_1\colon \overline{X}_B\rightarrow \bbP^1$ is a $\bbP^d$-bundle and the Brauer group of $\bbP^1$ is trivial, there exists a vector bundle $\sV$ such that $\overline{X}_B \cong \bbP(\sV)$.
	Let $\overline{\sF}=\rho^*\sF$.
	Then there is an induced injective  morphism $\nu^*\overline{\sF}\rightarrow T_{p_1}$   by Claim \ref{claim:pullback-sheaf-in-tangent}.
	By Theorem \ref{Proj-bundle}, there exists a numerically projectively flat subbundle $\sN\hookrightarrow \sV^*$ such that $$\nu^*\overline{\sF}\cong  p_1^*\sN\otimes\sO_{\bbP(\sV)}(1).$$
	We remark that  $\sN \cong \sL^{\oplus r}$ for some line bundle $\sL$ on $\bbP^1$. Hence, by replacing $\sV$ with $\sV\otimes \sL^*$, we may assume that $\sN\cong \sO_{\bbP^1}^{\oplus r}$.
	The nefness of $\nu^*\overline{\sF}$ then implies that $\sV$ is nef. In particular, there exist  integers $a_i\geqslant 0$ with $i=1,...,d+1-r$ such that
	\[
	\sV  \cong  \sN^*\oplus \bigoplus_{i=1}^{d+1-r}\sO_{\bbP^1}(a_i) \cong  \sO_{\bbP^1}^{\oplus r} \oplus\bigoplus_{i=1}^{d+1-r}\sO_{\bbP^1}(a_i).
	\]
	Note that $\sO_{\bbP(\sV)}(1)$ is globally generated.
	Let $g\colon \bbP(\sV)\rightarrow Y$ be the Iitaka fibration induced by $\sO_{\bbP(\sV)}(1)$. Then $g$ does not contract any curves contained in the fibers of $p_1$.
	
	Consider a complete curve $B'' \subseteq \bbP(\sV)$ contracted by $g\colon \bbP(\sV)\rightarrow Y$. Then $\sO_{\bbP(\sV)}(1)\vert_{B''}$ is trivial. As $f^*\sF \cong \sO_{\bbP(\sV)}(1)^{\oplus r}$, it follows that $(f^*\sF)\vert_{B''}$ is trivial. In particular, since $\sF$ is strictly nef, it follows that $B''$ is contracted by the composition
	\[
	f\colon \bbP(\sV) \cong \overline{X}_{B}\xrightarrow{\nu}\overline{X}\xrightarrow{\rho} X.
	\]
	By rigidity lemma, the morphism $f\colon \bbP(\sV)\rightarrow X$ factors through $g\colon \bbP(\sV)\rightarrow Y$.
	
	Let $C'' \subseteq \bbP(\sV)$ be the curve corresponding to $\nu^{-1}(C') \subseteq \overline{X}_B$.
	Then  $C''$ is not contracted by $g$ as  $f(C'')=C$ is again a curve. Let $l$ be a line contained in a general fiber of $p_1$,   $\widetilde{C} = g(C'')$ and  $\widetilde{l} = g(l)$. Note that $Y$  has Picard number $1$, thus $\widetilde{C}$ and $\widetilde{l}$ are numerically proportional in $Y$.
	Since $f$ factorizes through $g$, by   Lemma \ref{lemma:numerical-proportional-preseved}, we conclude that $f(l)$ is numerically proportional to $f(C'')=C$ in $X$.
\end{proof}

\subsubsection{Proof of Theorem \ref{lemma:Core-Technical-Theorem} in the case when $\sF\vert_F=\sO_{\bbP^d}(1)^{\oplus r}$}

\label{section:step-3-case-2}

If $\sF$ is a line bundle, by \cite[Corollaire]{Druel2004}, 
$X$ is isomorphic to $\bbP^n$. Thus we may assume that $r\geqslant 2$.

By Claim \ref{claim:extremal-ray-F=O1}, there exists a Mori contraction $\varphi\colon X\rightarrow T$ extending the   fibration $\varphi^\circ\colon X^\circ\rightarrow T^\circ$. By replacing $\overline{T}$   with a common resolution of $T'$ and $T$, we have the following commutative diagram:
\[
\begin{tikzcd}[column sep=large, row sep=large]
X^\circ \dar[swap]{\varphi^\circ}\rar &
X\dar[swap]{\varphi} & X' \lar{e}\dar[swap]{\varphi'}
& \overline{X} \lar{h} \arrow[bend right]{ll}[swap]{\rho} \dar{\overline{\varphi}}    &
Z\lar[swap]{q}\dar{\pi}
\\
T^\circ & T &
T' &
\overline{T} \lar \arrow[bend left]{ll}{\gamma} &
\overline{X}\lar{p=\overline{\varphi}}
\end{tikzcd}
\]
By Claim \ref{claim:pullback-sheaf-in-tangent}, there is an injective morphism $\rho^*\sF\rightarrow T_{\overline{X}/\overline{T}}$.
This induces an injective morphism $\sG = q^*\rho^*\sF\rightarrow T_{\pi}$.

By Theorem \ref{Proj-bundle},  there exists a numerically projectively flat subbundle $\sM$ of $\sE^*$ such that $\sG\cong  \pi^*\sM\otimes\sO_{\bbP(\sE)}(1)$.
Denote by $\pi_M\colon M\rightarrow \overline{X}$ the $\bbP^{r-1}$-bundle $\bbP(\sM^*)$. Then, by Theorem \ref{thm:num-proj-flat=proj-flat}, $M\rightarrow \overline{X}$ is isomorphic to a flat projective bundle  given by a representation $\pi_1(\overline{X})\rightarrow \PGL_{r}(\bbC)$. Moreover,  note that  $\sG\vert_{M}$ is isomorphic to $\pi_M^*\sM\otimes \sO_{\bbP(\sM^*)}(1)$ and the relative Euler sequence of $\bbP(\sM^*)$ induces a surjection $\sG\vert_{M}\rightarrow T_{M/\overline{X}}$. Let $P = (\rho\circ q)(M) \subseteq X$. According to Proposition \ref{prop:smothness-image-subbundle}, over $T^\circ$, $P$ is a $\bbP^{r-1}$-bundle. In particular, the general fiber of $P\rightarrow T$ is isomorphic to $\bbP^{r-1}$. Let $n\colon P'\rightarrow P$ be the normalization. Then we have a commutative diagram
\[
\begin{tikzcd}[column sep=large, row sep=large]
M\dar[swap]{\pi_M} \rar{\rho'} &  P'\dar{} \\
\overline{X} \rar& T.
\end{tikzcd}
\]
Note that $T$ has only $\bbQ$-factorial klt singularities and $n^*\sF$ is strictly nef with a surjection 
$$
\rho'^*n^*\sF\cong q^*\rho^*\sF\vert_M\rightarrow T_{M/\overline{X}},
$$
thus we can apply Proposition \ref{prop:equidimensionality-criterion-I} to conclude that $P'\rightarrow T$ is equidimensional.
Since $P'\rightarrow P$ is finite, $P\rightarrow T$ is again equidimensional.
In particular, by Proposition \ref{PROP:HN-degeneration-of-projective-spaces}, all the fibers of $P\rightarrow T$ are irreducible and generically reduced.

Let $t\in T$ and $F_t$  an irreducible component of the fiber of $\varphi\colon X\rightarrow T$ over $t$. Then there exists an irreducible component $B$ of $(\gamma\circ p)^{-1}(t)$ such that the induced morphism
\[
Z\times_{\overline{X}} B\rightarrow X
\]
is onto $F_t$. In particular, $F_t$ contains the fiber of $P\rightarrow T$ over $t$. Then Corollary \ref{cor:equi-dim} shows that $\varphi\colon X\rightarrow T$ is equidimensional. \hfil\qed

\subsubsection{Proof of Theorem  \ref{cor:main:part1}   in the case when $\sF\vert_F=\sO_{\bbP^d}(1)^{\oplus r}$}

\label{section:step-4-case-2}
We still use the notations of Section \ref{section:setup}. If $\sF$ is a line bundle, by \cite[Corollaire]{Druel2004}, $X$ is isomorphic to $\bbP^n$ and we are done. Thus we may assume that $r\geq 2$. Then Corollary \ref{cor-equidim-implies-bundle} shows that $\varphi\colon X\rightarrow T$ is a projective bundle between projective manifolds.

Next we assume that $\dim T>0$. Then we must have $r\geq 2$. Moreover, We may identify $T$ with $\overline{T}$ and $X$ with $\overline{X}$. In particular, $Z=\bbP(\sE)$ is isomorphic to the fiber product $X\times_T X$ and satisfies the following commutative diagram.
\[
\begin{tikzcd}[column sep=large, row sep=large]
X\dar[swap]{\varphi} & Z \lar{q}\dar[swap]{\pi} \\
T &
X\lar{p=\varphi}.
\end{tikzcd}
\]

Let $\sG  =  q^*\sF$. As explained in Section \ref{section:step-3-case-2}, there exists a numerically projectively flat subbundle $\sM$ of $\sE^*$ such that $\sG\cong \pi^*\sM\otimes \sO_{\bbP(\sE)}(1)$. In particular, $\sG$ is numerically projectively flat and so is $\sF$. Moreover, set $M=\bbP(\sM^*)$ and let $\pi_M\colon M\rightarrow X$ be the natural projection.  Then we have a surjection $\sG\vert_M\rightarrow T_{M/X}$ and $M\rightarrow X$ is isomorphic to a flat projective bundle given by a representation of $\pi_1(X)$ in $\PGL_{r}(\bbC)$. As $X\rightarrow T$ is a $\bbP^d$-bundle, $\pi_1(X)$ is isomorphic to $\pi_1(T)$.
Such a representation of $\pi_1(X)$ induces a flat $\bbP^{r-1}$-bundle  $Q\rightarrow T$ with the following commutative diagram
\[
\begin{tikzcd}[column sep=large, row sep=large]
Q\dar[swap]{} & M \lar{f}\dar[swap]{}    \\
T &
X\lar{p=\varphi}.
\end{tikzcd}
\] 

Let $P=q(M)$. By Proposition \ref{prop:smothness-image-subbundle}, we see that $P\rightarrow T$ is a $\bbP^{r-1}$-bundle. On the other hand, it is easy to see that an irreducible curve $C\subset M$ is contracted by $q$ if and only if it is contracted by $f$. Therefore, by rigidity lemma, $P$ is isomorphic to $Q$ as projective bundles over $T$. In particular, the pushforward of the surjection $\sG\vert_{M}\rightarrow T_{M/X}$ induces a surjection $\sF\vert_Q\rightarrow T_{Q/T}$.\hfill \qed

\vskip 2\baselineskip

\section{Proof of the hyperbolicity}
\label{section:hyperbolicity}
In this section, we  finish the proofs of Theorem \ref{thm:main-theorem}, Theorem  \ref{thm:simply-connected-Pn} and Corollary \ref{cor:existence-symmetric-forms}.
A projective manifold $Y$ is called \emph{Brody hyperbolic}  if every  holomorphic map  $f\colon \bbC\rightarrow Y$ is constant. Since $Y$ is assumed to be compact, the Brody hyperbolicity is equivalent to the Kobayashi hyperbolicity.  
The following lemma is an application of \cite[Theorem 1.1]{Yamanoi2010}, which reveals  the relationships between fundamental groups  and   degeneracy of entire curves.

\begin{lemma}
	\label{lemma:representation-degenerate-entire-curve}
	Let $Z$ be a projective variety. If there exists   a  subgroup  $G\subseteq \pi_1(Z)$ of finite index such that it  admits a linear representation whose image is not virtually abelian, then every holomorphic map $f\colon \bbC \to Z$ is degenerate, i.e. $f(\bbC)$ is not Zariski dense in $Z$.
\end{lemma}

\begin{proof}
	By taking a finite \'etale cover, it is enough to prove the case when $G=\pi_1(Z)$.
	Assume that there exists a holomorphic map $f\colon \bbC \to Z$ which is non-degenerate. Let $\overline{Z} \to Z$ be a desingularization.
	Then $f$ lifts to a holomorphic map $\overline{f}\colon \bbC \to \overline{Z}$. Since there is a surjective morphism $\pi_1(\overline{Z}) \to \pi_1(Z)$, we concluded that $\pi_1(\overline{Z})$ also admits a linear representation whose image is not virtually abelian. By \cite[Theorem 1.1]{Yamanoi2010}, $\overline{f}$ is degenerate and so is $f$. This is a contradiction.
\end{proof}

In order to deduce the hyperbolicity in Theorem \ref{thm:main-theorem} from Lemma \ref{lemma:representation-degenerate-entire-curve}, we need some preparatory results.

\begin{lemma}
	\label{lemma:representation-vir-abelian-not-strictly-nef}
	Let $Z$ be a positive dimensional projective variety.  Assume that $\sG$ is a flat vector bundle given by a linear representation   $\rho\colon \pi_1(Z) \to \mathrm{GL}_r(\bbC)$. If the image of $\rho$  is virtually abelian,  then $\sG$ is not strictly nef.
\end{lemma}

\begin{proof}
	Up to finite \'etale cover, we may assume that the image of $\rho$ is abelian.  Then the image $\rho(\pi_1(Z))$ can be simultaneously triangulated. Hence there is a quotient morphism of flat vector bundles $\sG \to \sL$ such that $\sL$ is a line bundle. Such a quotient induces a section $\sigma\colon Z \to \bbP(\sG)$. Moreover, $\sigma^*\sO_{\bbP(\sG)}(1) \cong \sL$. Since a flat line bundle is never strictly nef, this contradicts to the strict nefness of $\sG$.
\end{proof}

The following result  is a consequence of the Borel  fixed-point theorem.

\begin{prop}
	\label{prop:fixed-point}
	Let $G$ be a finitely generated virtually abelian subgroup of $\PGL_{d+1}(\bbC)$. Then there exists a subgroup $G'\subseteq G$ of finite index such that the natural action of $G'$ on $\bbP^d$ has a fixed point.
\end{prop}

\begin{proof}
	Since $G$ is virtually abelian, there exists a finite index subgroup $G'$ of $G$ such that $G'$ is abelian. Let $\overline{G'}$ be the Zariski closure of $G'$ in $\PGL_{d+1}(\bbC)$. Since $G'$ is finitely generated and $\PGL_{d+1}(\bbC)$ is a linear algebraic group, $\overline{G'}$ is  abelian.
	By replacing $G'$ with some subgroup of finite index if necessary, we can assume further that $\overline{G'}$ is connected. Therefore, by Borel  fixed-point theorem, the action  of $\overline{G'} $ on $\bbP^d$ has a fixed point.
\end{proof}

We obtain the following result from Lemma \ref{lemma:representation-vir-abelian-not-strictly-nef} and Proposition \ref{prop:fixed-point}.

\begin{prop}
	\label{prop:representation-non-abelian}
	Let $Z$ be a positive dimensional projective variety  and  $P\rightarrow Z$ a flat $\bbP^d$-bundle given by a representation $\rho\colon \pi_1(Z) \to \mathrm{PGL}_{d+1}(\bbC)$. Assume that  the relative tangent bundle $T_{P/Z}$ is strictly nef. Then  the image $\rho(\pi_1(Z))$ is infinite. Moreover, there exists   a  subgroup  $G\subseteq \pi_1(Z)$ of finite index such that it  admits a linear representation whose image is not virtually abelian.
\end{prop}

\begin{proof}
	Arguing by contraction, we assume that the image $\rho(\pi_1(Z))$ is finite. Then, after replacing $Z$ by a finite \'etale cover, the flat $\bbP^d$-bundle $P\rightarrow Z$ is isomorphic to $Z\times \bbP^d$. This contradicts to the strict nefness of $T_{P/Z}$.
	
	Next, if  the image of $\rho(\pi_1(Z))$ is not  virtually abelian, then we are done. Otherwise, if $\rho(\pi_1(Z))$ is  virtually abelian,  we shall construct some $G\subseteq \pi_1(Z)$ as required. Indeed, by Proposition \ref{prop:fixed-point}, there is a   subgroup $G'$ of $G$ of finite index such that the natural action of $G'$ on $\bbP^d$ has a fixed point.  Let $\Gamma \subseteq \pi_1(Z)$ be the preimage $\rho^{-1}(G')$. Then, by replacing $Z$ with the finite \'etale cover induced by $\Gamma \subseteq \pi_1(Z)$, we may assume that the   natural action of $G$ on $\bbP^d$ has a fixed point $p$.
	
	We remark that $P=(\bbP^d\times \widetilde{Z})/\pi_1(Z)$ where $\widetilde{Z}$ is the universal cover of $Z$, and  the action of $\pi_1(Z)$ is defined as $$ g. (x,z) = (\rho(g)(x), g\cdot z) $$ for every $g\in \pi_1(Z)$.
	Let $W=(\{p\}\times \widetilde{Z})/\pi_1(Z)$.
	Then $W\cong Z$ and there is a closed embedding  $W\hookrightarrow P$.
	We note that, the restriction $T_{P/Z}|_W$ is just equal to the bundle $(T_{\bbP^d,p}\times \widetilde{Z})/\pi_1(Z)$, where the action of $\pi_1(Z)$ on $T_{\bbP^d,p}$ is the differentiation of the action of $\pi_1(Z)$ on $\bbP^d$ at $p$.
	Consequently, $T_{P/Z}|_W$ is a flat vector bundle on $W$.
	It is strictly nef as well. This implies that the representation of $\pi_1(Z)$ corresponding to $T_{P/Z}|_W$ is not virtually abelian by Lemma \ref{lemma:representation-vir-abelian-not-strictly-nef}, and we are done.
\end{proof}

As an application, one can   derive the following corollary.

\begin{cor}
	\label{cor:fundamental-group-subvariety}
	Under the assumption of Theorem \ref{thm:main-theorem}, let $\varphi\colon X\rightarrow T$ be the $\bbP^d$-bundle structure provided in Theorem \ref{cor:main:part1}. If $Z$ is a  subvariety of $T$ with $\dim Z>0$, then there exists a finite index subgroup of $\pi_1(Z)$ admitting a linear representation whose image is not virtually abelian. In particular, every holomorphic map $f:\mathbb{C}\rightarrow Z$ is degenerate.
\end{cor}

\begin{proof}
	By Theorem \ref{cor:main:part1}, there is always a flat projective bundle $Q$ over $T$ such that the relative tangent bundle $T_{Q/T}$ is strictly nef. Hence,   $P=Q\times_T Z$ is a flat projective bundle over $Z$. Furthermore, the relative tangent bundle $T_{P/Z}$ is strictly nef as well.
	Therefore, by  Proposition \ref{prop:representation-non-abelian}, some   subgroup of $\pi_1(Z)$ of finite index admits a linear representation whose image is not virtually abelian. By Lemma \ref{lemma:representation-degenerate-entire-curve}, every holomorphic map $f:\bbC\rightarrow Z$ is degenerate.
\end{proof}

\vskip 1\baselineskip

\begin{proof}[{Proof of Theorem \ref{thm:main-theorem}}]
	Thanks to Theorem \ref{cor:main:part1}, we only need  to prove the hyperbolicity of $T$.
	Let  $f\colon \bbC\rightarrow T$ be a holomorphic map. 
	Assume by contradiction that $f$ is not a constant map. Let $Z$ be the Zariski closure of $f(\bbC)$. Then $\dim Z >0$ and the induced map $\bbC \to Z$ is non degenerate. 
	This contradicts to Corollary \ref{cor:fundamental-group-subvariety}.
\end{proof}

\begin{proof}[{Proof of Theorem  \ref{thm:simply-connected-Pn}}]
	By Theorem \ref{thm:main-theorem}, there is a $\bbP^d$-bundle structure $\varphi\colon X\to T$. In particular, $\pi_1(T)$ is virtually abelian. By Corollary \ref{cor:fundamental-group-subvariety}, we deduce that $\dim T=0$ and consequently $X$ is isomorphic to $\bbP^n$.
\end{proof}

\begin{proof}[Proof of Corollary \ref{cor:existence-symmetric-forms}]
	Since $X$ is not isomorphic to $\bbP^n$, by Theorem \ref{thm:main-theorem} and Theorem \ref{cor:main:part1}, there is a flat projective bundle $Q\rightarrow T$ such that $\dim T>0$ and $T_{Q/T}$ is strictly nef. Hence by Proposition \ref{prop:representation-non-abelian}, there is a linear representation of $\pi_1(T)$ with infinite image. Since $\pi_1(X) \cong \pi_1(T)$,  we obtain  the existence of nonzero symmetric differentials on $X$ by using \cite[Theorem 0.1]{BrunebarbeKlinglerTotaro2013}.
\end{proof}

\vskip 1\baselineskip

\def\cprime{$'$} 

\renewcommand\refname{Reference}
\bibliographystyle{alpha}
\bibliography{strcitlynef}

\end{document}